\documentclass[11pt,reqno]{amsart}

\usepackage{dsfont,tikz, float, amsmath, amssymb ,amsthm}
\usepackage[a4paper,hmarginratio={1:1}]{geometry}
\usepackage[english]{babel}
\usetikzlibrary{plotmarks}

\newcommand{\scal}[2]{ \langle #1 , #2 \rangle}
\newcommand{\ct}{\mathrm{ct}}

\renewcommand{\Re}{\mathrm{Re}}
\renewcommand{\Im}{\mathrm{Im}}

\def\N{\mathbb N}
\def\Z{\mathbb Z}
\def\R{\mathbb R}
\def\C{\mathbb C}

\numberwithin{equation}{section}

\newtheorem{theorem}{Theorem}[section]
\newtheorem{lemma}[theorem]{Lemma}

\theoremstyle{remark}
\newtheorem*{remark}{Remark}

\begin{document}
\title[Spectral asymptotics for the Dirichlet Laplacian  with a Neumann window]{Spectral asymptotics for the Dirichlet Laplacian  with a Neumann window via a Birman-Schwinger analysis of the Dirichlet-to-Neumann operator}
\author{Andr\'e H\"anel}
\address{Andr\'e H\"anel, Institut  f\"ur Analysis, Leibniz Universit\"at Hannover, Welfengarten 1,  D-30167 Hannover}
\email{andre.haenel@math.uni-hannover.de}

\author{Timo Weidl}
\address{Timo Weidl, Institut f\"ur Analysis, Dynamik und Modellierung, Universit\"at Stuttgart, Pfaffenwaldring 57, D-70569 Stuttgart}
\email{weidl@mathematik.uni-stuttgart.de} 

\begin{abstract}
In the  present article we will give a new proof of the ground state  asymptotics of the Dirichlet Laplacian with a Neumann window acting on functions which are defined on a two-dimensional infinite strip or a three-dimensional
infinite layer. The proof is based on the  analysis of the corresponding Dirichlet-to-Neumann operator as a first order classical pseudo-differential operator. Using the explicit representation of its symbol we  prove  an asymptotic expansion
as the window length decreases. 
\end{abstract}
\maketitle

\section{Introduction}
In what follows we consider an infinite quantum waveguide subject to a perturbation of the boundary conditions. In spectral theory this type of perturbation is of particular interest, since  it is  non-additive and may not be treated with standard  methods,
such as the Birman-Schwinger principle. The  simplest case arises by considering the Dirichlet Laplacian on an infinite strip 
having a so-called Neumann window. Let $\Omega = \R \times (0,\alpha)$. We consider $-\Delta$ on $\Omega$ with Dirichlet boundary on all of $\partial \Omega$ except for some small part of the boundary,
where we impose Neumann boundary conditions. We are interested in the behaviour of the discrete eigenvalues below the essential spectrum $[\pi^2/\alpha^2,\infty)$ depending on the length of the window, 
cf.\ Figure \ref{fig:Neumann_window}.
\begin{figure}[ht] 
\begin{center}
\begin{tikzpicture}
\begin{scope}[thick,font=\scriptsize]
  	\draw (-3.5,0.8) -- (3.5,0.8);
  	\draw (-3.5,2.5) -- (-1.1,2.5);
  	\draw (1.1,2.5) -- (3.5,2.5);
	\draw (-2.4, 2.5) node[anchor=south] {$u=0$};
	\draw (2.4,2.5)  node[anchor=south] {$u=0$};
	\draw (0,2.5)  node[anchor=south] {$\partial_n u   = 0$};
	\draw [|<->|] (-1.1,2.5) -- (1.2,2.5);%
	\draw (0,2.4)   node[anchor=north] {$2\ell$};
	\draw (1.2,1.5) node {$-\Delta u = \lambda u   $};
	\draw [<->] (4,0.75) -- (4,2.5);%
	\draw (-2.8,0.8) -- (-2.8,0.8) node[anchor=north] {$u=0$};  
	\draw (4.2,1.625)  node[anchor=west] {width=$\alpha$}; 
\end{scope}
\end{tikzpicture}
\end{center}
\caption{The Dirichlet Laplacian with a Neumann window.}
\label{fig:Neumann_window}
\end{figure}
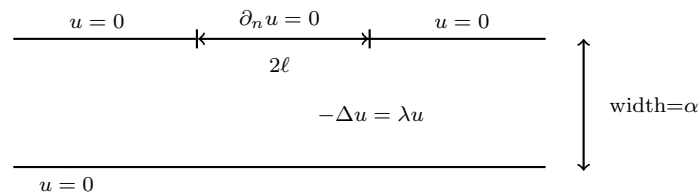

This case  was first investigated in \cite{ExnerSebaTaVa},
where the  existence of an eigenvalue was proved by  a variational argument. Moreover, a numerical computation given by these authors  suggested that for small windows of size $2\ell$ the 
distance of the eigenvalue to the spectral threshold $\pi^2/\alpha^2$ is of order $\ell^4$.
The first analytic proof concerning  this fact  was given by Exner and Vugalter in \cite{ExVu}. They proved  a two-sided asymptotic estimate,  i.e., for small $\ell >0$ there exist a unique eigenvalue
$\lambda(\ell)$ below the essential spectrum $[\pi^2/\alpha^2, \infty)$ and constants  $c_1, c_2 > 0$ such that 
\begin{equation}\label{eq:two_sided} 
	c_1 \ell^4 \le \pi^2/\alpha^2 - \lambda(\ell) \le c_2 \ell^4   \qquad \text{as} \quad \ell \to 0 . 
\end{equation}
In fact in  \cite{ExnerSebaTaVa,ExVu} the authors considered the more general case of two quantum waveguides which are coupled through a small window, cf.\ 
Figure \ref{fig:waveguides}. 
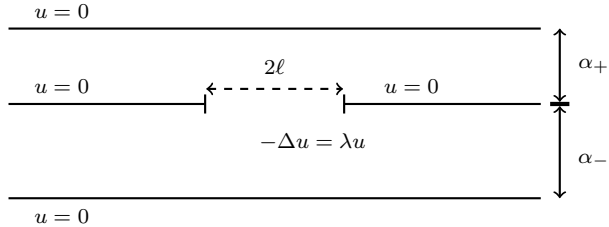
\begin{figure}[ht] 
\begin{center}
\begin{tikzpicture}
\begin{scope}[thick,font=\scriptsize]
	\draw (-3.5,0.25) -- (3.5,0.25);
 	\draw (-3.5,2.5) -- (3.5,2.5);
 	\draw [-|] (-3.5,1.5) -- (-0.9,1.5);
 	\draw [-|] (3.5,1.5) -- (0.9,1.5);
 	\draw [|<->] (3.75, 1.5) -- (3.75,2.5);
 	\draw [<->|] (3.75,0.25) -- (3.75,1.5);
	\draw (4.2,2)  node {$\alpha_+$}; 
	\draw (4.2,0.75)  node {$\alpha_-$}; 
	\draw (0.5, 1) node {$-\Delta u = \lambda u $};
	\draw (-2.8, 2.5) node[anchor=south] {$u=0$};
	\draw (-2.8, 1.5) node[anchor=south] {$u=0$};
	\draw (1.8, 1.5) node[anchor=south] {$u=0$};
	\draw (-2.8, 0.25) node[anchor=north] {$u=0$};
	\draw [<->, dashed] (-0.9,1.7) -- (0.9,1.7);
	\draw (0, 1.75) node[anchor=south] {$2 \ell$};
\end{scope}
\end{tikzpicture}
\caption{Two laterally coupled quantum waveguides of width $\alpha_+$ and  $\alpha_-$.}
\label{fig:waveguides}
\end{center}
\end{figure} 

If  both waveguides have the same width  $\alpha_+ = \alpha_-  = \alpha$, then we may use the symmetry with respect to the horizontal direction. In this case the eigenvalue problem
is equivalent to the mixed problem  in Figure \ref{fig:Neumann_window}. 
The  proof of the two-sided asymptotic estimate in \cite{ExVu}  is based on a variational argument.  
The upper bound may easily be obtained using  a suitable test-function and 
the  min-max principle for self-adjoint operators. However, the more delicate part consists in finding a uniform  lower bound for the variational coefficient.
In order to prove such an estimate Exner and Vugalter decomposed  an arbitrary test-function, 
using  an expansion in the vertical direction. 

 Popov \cite{Popov99} refined the two-sided estimate and proved that the ground state  eigenvalue $\lambda(\ell)$
satisfies the following asymptotic behaviour
\begin{equation} \label{eq:asympt_2D_Popov}
	 \frac{\pi^2}{\alpha^2} - \lambda(\ell) = \begin{cases} \displaystyle \left( \frac{\pi^3}{4\alpha^3} \right)^2 \ell^4 + o(\ell^4),
	& \alpha_+ \neq \alpha_-, \\[12pt]  \displaystyle 
	\left( \frac{\pi^3}{2\alpha^3} \right)^2 \ell^4 + o(\ell^4) , & \alpha_+ = \alpha_-  ,\end{cases} \qquad \text{as} \quad \ell \to 0 . 
\end{equation}
His proof is based on a scheme which matches  the asymptotic expansions for the eigenfunctions, cf.\ \cite{Ilin, Gad92}.
Popov uses  different  expansions for the  eigenfunctions near the  window and distant from the window. 
Using  the explicit formulae for the Green's functions in the upper and lower waveguide  he computes the asymptotic behaviour of the eigenvalue. 
Further terms in this  expansion have been calculated in  \cite{FrolovPopov00}. 
In \cite{Popov02} the approach was generalised to three-dimensional layers, cf.\ also   \cite{ExVuLayers} for a two-sided asymptotic estimate. Further extensions include 
e.g.\ higher dimensional cylinders  \cite{Popov201, Gad04}, 
a finite or an infinite  number of windows \cite{Popov201,Popov101,BorisovBunoiuCardone10, BorisovBunoiuCardone11, BorisovBunoiuCardone13,Nazarov13}, 
the case of three coupled waveguides \cite{FrolovPopov03} or magnetic 
operators \cite{BoEkKov05}. The case of two retracting distant windows has been investigated in \cite{BorisovExner04, BorisovExner06}. For an overview concerning spectral problems 
in quantum waveguides we refer to \cite{ExnerKov15}.

We provide a new approach for the symmetric case which uses the  explicit  representation  of  the Dirichlet-to-Neumann operator. This allows us to reformulate
the singular perturbation of the original operator into an additive perturbation of the Dirichlet-to-Neumann operator, or merely its truncated part. 
We replace the matching scheme for
 the eigenfunctions in \cite{Popov99} by an asymptotic expansion of the Dirichlet-to-Neumann operator and a subsequent use of the Birman-Schwinger principle. 
As a particular consequence we will observe that only the principal  symbol has an influence on  the first term of the asymptotic formula. 
In a similar way we treat  the case of two coupled quantum waveguides. An application of the method to elastic waveguides may be found in \cite{HaenelWeidl}. 
\subsection*{Structure of the article} We start by treating the two-dimensional case. We consider the Laplacian $-\Delta$ on $\Omega = \R \times (0, \alpha)$ with   
Dirichlet boundary conditions except for some  small set $\Sigma_\ell \times \{ 0 \}  \subseteq \partial \Omega$, where we impose Robin boundary conditions.
Here $\Sigma_\ell := \ell \cdot \Sigma \subseteq \R$ and   $\Sigma \subseteq \R$ is a finite union of bounded open intervals.
Section 2 starts with the definition of the self-adjoint realisation of the corresponding Laplacian and the introduction of the Dirichlet-to-Neumann operator and of the  Dirichlet-to-Robin operator. 
The asymptotic formula for the eigenvalue of the corresponding Laplacian is stated and proven in Section \ref{sec:2D}, see Theorems  \ref{th:main_2D} and \ref{th:main_2D_special}. Additionally, in Theorems  \ref{th:main_2D} and \ref{th:main_2D_special}
we prove the uniqueness of  the eigenvalue for small window sizes and in Theorem \ref{th:main_coupled} we treat the case of two quantum waveguides coupled through a small window. 

Section \ref{sec:3D} is devoted to  three-dimensional layers of the form  $\Omega = \R^2 \times (0, \alpha)$. In this case the  Robin window is given by $\Sigma_\ell \times \{ 0 \} \subseteq \partial \Omega$,
where $\Sigma_\ell := \ell \cdot \Sigma  \subseteq \R^2$ is a bounded open set with Lipschitz boundary, cf.\ Figure \ref{fig:layer}. We follow the same scheme as in the two-dimensional
case and prove in Theorem \ref{th:main_3D}  an  asymptotic formula for  the ground state eigenvalue  as the window length decreases. 
\begin{figure}[H]
\begin{center}
\begin{tikzpicture}[thick,font=\scriptsize]
	\draw[dotted] (0,0) -- (0,1.5);
	\draw[dotted] (4.5,0) -- (4.5,1.5);
	\draw[dotted] (2,1) -- (2,2.5);
	\draw[dotted] (6.5,1) -- (6.5,2.5);
	\draw (0,0) -- (4.5,0) -- (6.5,1);
	\draw (0,0) -- (4.5,0) -- (6.5,1) -- (2,1) -- (0,0);
	\draw (0,1.5) -- (4.5,1.5) -- (6.5,2.5) -- (2,2.5) -- (0,1.5);
 	\begin{scope}[xshift=95, yshift=62, xscale=0.6, yscale=0.4] 
  	\draw[draw, fill=white,rotate=180] 
		 (-0.6,0).. controls (0.18259,0.02218) and (0.31251,0.21639)
		..(0.5,0.2).. controls (0.68426,0.18389) and (0.80980,-0.04411)
		..(1,0.5).. controls (1.19307,0.70254) and (1.18126,0.93079)
		..(1,1).. controls (0.83172,1.06425) and (0.67339,0.91497)
		..(0.5,0.9).. controls (0.32763,0.88512) and (0.17237,1.00348)
		..(0,1).. controls (-0.74957,0.98486) and (-0.64760,-0.07866)
		..(-0.6,0);
	\end{scope}
\end{tikzpicture}
\caption{An  infinite layer with  a small window.}
\label{fig:layer}
\end{center}
\end{figure}
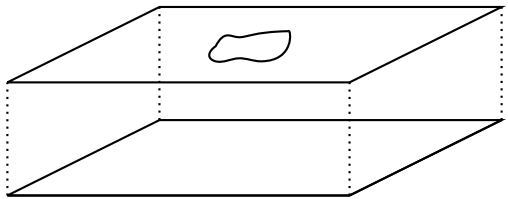

\section{The two-dimensional case}\label{sec:2D}
\subsection*{The construction of the Dirichlet-to-Robin  operator}
Let $\alpha > 0$ and put $\Omega := \R \times (0,\alpha)$ with coordinates $(x,y) \in \Omega$. Let $\Sigma \subseteq \R$ be the  Robin window and assume that $\Sigma$ is a bounded open set, which is a finite union of open
intervals. We denote the scaled window  by $\Sigma_\ell := \ell \cdot \Sigma$. 
The Laplacian on $\Omega$ with Robin boundary conditions on 
$\Sigma_\ell \times \{0\}$ and Dirichlet boundary conditions on the remaining part of the boundary is defined by the 
quadratic form
\begin{equation}
	a_{\ell,b} [u] :=  \int_{\Omega} |\nabla u(x,y)|^2  \ \mathrm{d} x \ \mathrm{d} y + \int_{\Sigma_\ell}  b (x) \cdot |u(x,0)|^2  \ \mathrm{d} x ,
\end{equation}
with the form domain 
\begin{equation}
	D[a_{\ell,b}]  := \{ u \in H^1(\Omega) :   u|_{\R \times \{\alpha\}}  = 0 \;  \wedge \; \mathrm{supp}(u|_{\R \times \{0\}} ) \subseteq  	
	\overline{\Sigma_\ell}\} .
\end{equation} 
Here $b \in L_\infty(\R)$ is a real-valued  function and $u|_{\R \times \{0\}}  ,  u|_{\R \times \{\alpha\}}   \in H^{1/2}(\R)$ are the boundary traces of the function $u \in H^{1}(\Omega)$. 
Then $a_{\ell,b} $ defines a closed semi-bounded form on $L_2(\Omega)$ and gives rise to a self-adjoint operator,
which we denote by $A_{\ell,b}$. The essential spectrum of the operator $A_{\ell,b}$ is given by 
 $\sigma_{\mathrm{ess}}(A_{\ell,b}) = [\pi^2/\alpha^2, \infty)$. This well-known fact is due to Birman \cite{Birman62}, where he gives a proof in the case
of compact boundary $\partial \Omega$.

As a first step we consider  the Dirichlet-to-Neumann operator acting on the lower part of the boundary $\R \times \{0\}$. 
For $s \in \R$ let $H^s(\R)$ be the 
standard Sobolev space on $\R$ with the usual norm defined  via Fourier transform.  
Let  $\omega \in \C$ and $g \in H^{1/2} (\R)$. We consider a weak  solution 
$u \in H^1(\Omega)$ of the Poisson  problem 
\begin{equation}\label{eq:Poisson_bvp}
  (- \Delta -  \omega) u = 0  \quad \text{in } \Omega ,  \qquad u|_{\R \times \{0\}}   =  g, \qquad  u|_{\R \times \{\alpha\}}   = 0 .
\end{equation}
 Applying the Fourier transform in the horizontal direction, 
 it follows from \eqref{eq:Poisson_bvp} that $\hat u (\xi,y)$ solves 
\begin{equation}\label{eq:Poisson_FT}
	( - \partial^2_y  + \xi^2 - \omega) \hat u(\xi,y) = 0, \qquad 
	\hat u (\xi,0) = \hat g(\xi) , \qquad   \hat u (\xi,\alpha) = 0  
\end{equation}
with  $(\xi,y) \in \R \times (0,\alpha)$. 
Conversely, if   $u \in H^1(\Omega)$ is given such that its Fourier transform $\hat u(\xi,y)$ solves the family \eqref{eq:Poisson_FT} of Sturm Liouville problems,
then $u$ is a weak solution of the Poisson problem \eqref{eq:Poisson_bvp}. For fixed $\xi \in \R$ with $\xi^2 \neq \omega$, the solution of \eqref{eq:Poisson_FT} is given by
\begin{equation}\label{eq:sol_u}
	\hat u(\xi, y) = \frac{\hat g (\xi)}{\sinh
 	(\alpha  \sqrt{\xi^2 - \omega})} \cdot \sinh (( \alpha - y)\sqrt{\xi^2 - \omega}) .
\end{equation}
Here and subsequently we choose the
 branch  of the square root function such that 
 $z \mapsto \sqrt{z}$ is holomorphic in $\C \backslash (-\infty,0]$ and 
such that  $\sqrt{z} > 0 $ for  $z > 0 $. Moreover, we extend the definition to $z \in (-\infty, 0 ]$
 and assume  that  $\Im \left( \sqrt{z} \right) \ge  0 $ for $z \le  0$.  Actually, the expression for $\hat u$ is independent
  of the value of the square root function as long as one uses the same 
in  the two terms. 
\begin{lemma}
	\label{lemma:Poisson_op}
	Let $\omega \in \C \backslash [\pi^2/\alpha^2, \infty)$. For every  $g \in H^{1/2}(\R)$
	there exists a unique $u \in H^1(\Omega)$ which solves \eqref{eq:Poisson_bvp}, and
	moreover $\| u \|_{H^1(\Omega)} \le c \| g\|_{H^{1/2}(\R)}$ with $c = c(\omega, \alpha) > 0$ independent of $g$.
\end{lemma}
For the  proof of  Lemma \ref{lemma:Poisson_op} one has to  verify that the function 
$u$ given by \eqref{eq:sol_u} belongs indeed to $H^1(\Omega)$ if $g \in H^{1/2}(\R)$. 
We want to omit this simple calculation. 
\begin{remark}
	If $\omega \ge \pi^2/\alpha^2$, then in general $\hat u$  will have a singularity
	and the above mapping property does not hold true. This is to be expected, as
	in this case $\omega$ will be located in the essential spectrum $[\pi^2/\alpha^2, \infty)$.
\end{remark}
Here and subsequently we assume that    $\omega \in \C \backslash[\pi^2/\alpha^2, \infty)$. 
Let $g \in H^{1/2}(\R)$ and let $u$ be the solution of the Poisson problem \eqref{eq:Poisson_bvp}. 
Its normal derivative $\partial_n u$ satisfies 
$$\widehat{\partial_n u} (\xi, 0 ) = m_\omega(\xi) \hat g (\xi),$$
where
\begin{equation}
  	m_\omega(\xi) := \sqrt{\xi^2 - \omega} \cdot \coth (\alpha \sqrt{\xi^2-\omega} ) .
\end{equation}
The Dirichlet-to-Neumann operator is defined by $D_\omega : H^{1/2}(\R) \to  H^{-1/2}(\R)$, 
\begin{equation}
   	 \widehat{D_\omega g} (\xi) 
   		:= \hat g (\xi) \cdot \sqrt{\xi^2 - \omega} \cdot \coth (\alpha \sqrt{\xi^2-\omega} ) . 
\end{equation}
We note that $D_\omega$ 
is a classical pseudo-differential operator of order $1$ with $x$-independent 
symbol $m_\omega$. 
Since the operator $A_{\ell,b}$ is defined by its quadratic form we give  a variational  characterisation of  the Dirichlet-to-Neumann operator 
$D_\omega$. We also refer to Chapter 4 in McLean \cite{McLean} for mixed problems  formulated  in their variational form. 
\begin{lemma}\label{lemma:DtNo_variational}
	Let $\omega \in \C \backslash [\pi^2/\alpha^2 ,\infty)$ and   $g \in H^{1/2}(\R)$. We denote by 
	 $u \in H^1(\Omega)$ the solution of the Poisson problem  \eqref{eq:Poisson_bvp}. Then for $h \in H^{-1/2}(\R)$ 
	the following two assertions are equivalent:
	\begin{enumerate}
		\item $ h = D_\omega g$.
		\item For all $v \in H^1(\Omega)$ with  $v |_{\R \times \{\alpha\}}= 0$  we have
		\begin{equation}\label{eq:weak_def_DtoN_2D}  \scal{\nabla u}{\nabla v}_{\Omega} =  \omega \scal{u}{v}_{\Omega}  +
			\scal{h}{v|_{\R \times \{0\}}}_{\R}  		.\end{equation}
	\end{enumerate}
	Here  $\scal{\cdot }{\cdot }_{\Omega}$ and  $\scal{\cdot }{\cdot }_{\R}$ denote the  dual pairings
	with respect to the scalar product in $L_2(\Omega)$ and $L_2(\R)$ identified with $L_2(\R \times \{0\})$.
\end{lemma}
\begin{proof}
	Let $g \in H^{1/2}(\R)$ and $u \in H^1(\Omega)$ be chosen as above. From \eqref{eq:sol_u} and 
	integration by parts we obtain
	\begin{align*} 
		\scal{\nabla u}{\nabla v}_{\Omega} - \omega \scal{u}{v}_{\Omega}  &= \int_\R \partial_y \hat u(\xi , 0) \overline{\hat v(\xi, 0)} \ \mathrm{d} \xi \\
		&= \int_\R \sqrt{\xi^2 - \omega} \cdot \coth(\alpha \sqrt{\xi^2 - \omega}) \hat g(\xi) 
			\overline{\hat v(\xi, 0)} \ \mathrm{d} \xi \\
		&= \scal{D_\omega g}{v|_{\R \times \{0\}}}_\R . 
	\end{align*}
	This proves one direction of the equivalence. The converse follows as the trace operator $
	u \mapsto u|_{\R \times \{0\}} : H^{1}(\Omega) \to H^{1/2}(\R)$ has a continuous right inverse,
	cf.\ \cite[Lemma 3.36]{McLean}. In particular, 	for every $f \in H^{1/2}(\R)$ there exists 
	$v \in H^1(\Omega)$ such that $v|_{\R \times \{\alpha\}} = 0$ and 
	$v|_{\R \times \{0\}} =f$, and thus,  $D_\omega g$ is uniquely defined by \eqref{eq:weak_def_DtoN_2D}.
\end{proof}
In order to treat the mixed boundary value problem we introduce for $s \in \R$ the following function spaces
\begin{align}\label{def:Hs_0}
	\tilde H_0^{s}(\Sigma_{\ell})  &:= \{ g \in H^{s}(\R) : \mathrm{supp}(g) \subseteq \overline{\Sigma_{\ell}}  \} , \\
	H^s(\Sigma_{\ell}) &:=  \{ g \in (C_c^\infty(\Sigma_\ell))' : \exists G \in H^{s}(\R) \text{ with }  g = G|_{\Sigma_{\ell}} \} .
	\label{def:Hs} 
\end{align}
Here $C_c^\infty(\Sigma_\ell)$ is the space of  smooth functions with compact support in   $\Sigma_\ell $;
we denote by  $(C_c^\infty(\Sigma_\ell))'$ the space of distributions on $\Sigma_\ell$. 
We note that $\tilde H_0^{s}(\Sigma_\ell)$ is a closed subspace of distributions in $\R$ whereas
$H^{s}(\Sigma_\ell)$ is a subspace  of distributions in  $\Sigma_\ell$. The latter
space may be identified with the quotient space
$$ \raisebox{0.5ex}{\ensuremath{H^s(\R)}}
\ensuremath{\mkern-0.3mu}/\ensuremath{\mkern-0.3mu}
\raisebox{-0.5ex}{\ensuremath{\tilde H^s_0(\R \backslash \overline{\Sigma_\ell}) }} , $$
where $\tilde H^s_0(\R \backslash \overline{\Sigma_\ell})$ contains, by definition, those distributions 
in $H^s(\R)$ which have support in $\R \backslash \Sigma_\ell$. We endow the spaces in \eqref{def:Hs_0} and 
\eqref{def:Hs} with their natural topology, i.e., $\tilde H^s_0(\Sigma_\ell)$ carries  the subspace topology of $H^s(\R)$ and 
$H^s(\Sigma_\ell)$ has the quotient topology. For $s \ge   0$ we may identify $\tilde H_0^{s}(\Sigma_\ell)$ 
with the subspace of $L_2(\Sigma_\ell)$ which consists of 
those functions whose extension by $0$ yields an element of $H^s(\R)$. Furthermore, the 
space  $\tilde H^s_0(\Sigma_\ell)$ is an isometric realisation of the (anti-)dual of
$H^{-s}(\Sigma_\ell)$ and vice-versa. The dual pairing is given by the expression
\begin{equation}\label{eq:dual_pairing}
	\scal{g}{h}_{\Sigma_\ell} := \scal{G}{h}_{\R}  ,\qquad g \in H^{-s}(\Sigma_\ell) , \; h \in \tilde H^s_0(\Sigma_\ell) ,
\end{equation}
where $G \in H^{-s}(\R)$ denotes any extension of $g$, cf.\ \cite[Theorem 3.14]{McLean}. Note that   $C_c^\infty(\Sigma_\ell)$ 
is a dense subset of $\tilde H^{1/2}_0(\Sigma_\ell)$, cf. \cite[Theorem 3.29]{McLean}. In particular the above 
expression is independent of the chosen extension $G$. 
Thus, the  domain of the quadratic form $a_{\ell,b}$ may 
be rewritten  as  follows
\begin{equation}
    D[a_{\ell,b}] =  \left\{ u \in H^{1}(\Omega ) :    u|_{\R \times \{ \alpha\} } = 0 \; \wedge  \; 
	u|_{\R \times \{0\} } \in \tilde H^{1/2}_{0}(\Sigma_\ell) \right \} .
\end{equation}
We define  the truncated Dirichlet-to-Neumann operator 
\begin{equation}
	D_{\ell,\omega} : \tilde H^{1/2}_{0}(\Sigma_\ell) \to H^{-1/2} (\Sigma_\ell), \qquad  
	D_{\ell,\omega} := r_\ell \, D_\omega e_\ell  , \end{equation}
where 
$$ r_\ell : H^{-1/2}(\R) \to H^{-1/2}(\Sigma_\ell)$$ 
is the restriction operator and 
$$e_\ell : \tilde H^{1/2}_{0}(\Sigma_\ell) \to H^{1/2}(\R)$$
is the embedding. Identifying $ \tilde H^{1/2}_{0}(\Sigma_\ell) $
with a subspace of $L_2(\Sigma_\ell)$, the operator $e_\ell$ is simply extension by $0$. 
Considering the corresponding topologies 
one easily observes that $D_{\ell,\omega}$ is a bounded linear 
operator. Recalling  that $b\in L_\infty(\R)$, we define the  truncated Dirichlet-to-Robin operator by 
\begin{equation}
	D_{\ell,\omega} + b : \tilde H^{1/2}_{0}(\Sigma_\ell) \to H^{-1/2} (\Sigma_\ell) , \qquad D_{\ell, \omega} + b := r_\ell  ( D_{ \omega} + b )  e_\ell , 
\end{equation}
where we identify $b$ with the corresponding multiplication operator.
The next lemma gives a characterisation of the eigenvalues of $A_{\ell,b}$ in terms of the truncated
Dirichlet-to-Robin operator. 
\begin{lemma}\label{lemma:kernel}
	Let $\omega \in \C \backslash [\pi^2/\alpha^2, \infty)$ and  $\ell > 0$. 
	Then
	$$ \mathrm{dim}\; \mathrm{ker} (A_{\ell,b} - \omega) = \mathrm{dim}\; \mathrm{ker} (D_{\ell,\omega} + b) . $$
\end{lemma}
\begin{proof}
	The assertion follows if we prove that the trace mapping is an isomorphism of $\mathrm{ker} (A_{\ell,b} - \omega)$
	onto  $\mathrm{ker} (D_{\ell,\omega} + b)$. Let us first  prove that it indeed maps  $\mathrm{ker} (A_{\ell,b} - \omega)$
	into $\mathrm{ker} (D_{\ell,\omega} + b)$. Let $u \in \ker (A_{\ell,b}- \omega)$ 
	and denote $g := u|_{\R \times \{0\}} \in \tilde H^{1/2}_{0}(\Sigma_\ell) $
	its boundary trace. 
	Let  
	$h \in 	\tilde H^{1/2}_{0}(\Sigma_\ell)$ be an arbitrary test function and choose $v\in D[a_{\ell,b}]$, 
	such that $v|_{\R \times \{0\}} = h$. 
	The dual pairing \eqref{eq:dual_pairing} and  Lemma \ref{lemma:DtNo_variational} imply
	\begin{align*} \scal{
	    ( D_{\ell,\omega} + b ) g }{h}_{\Sigma_\ell } &= \scal{D_\omega g + b g }{h}_{\R } 
  		= \scal{\nabla u}{\nabla v}_{\Omega} + \scal{bg}{h}_\R -   \omega \scal{u}{v}_{\Omega} \\
	    &= a_{\ell, b} [u,v] - \omega \scal{u}{v}_\Omega = 0 ,
	\end{align*}
	as $D_\omega g \in H^{1/2}(\R)$ is obviously an extension of $D_{\ell,\omega} g \in H^{1/2}(\Sigma_\ell)$
	and $u$ is an eigenfunction for the eigenvalue $\omega$. 
	Hence,  $g \in \mathrm{ker}( D_{\ell,\omega} + b)$, which proves that the mapping
	$$ \mathrm{ker}( A_{\ell,b} - \omega)  \ni u \mapsto u|_{\R \times \{0\}} \in \mathrm{ker} ( D_{\ell,\omega} + b) $$
	is  well defined. Moreover, Lemma \ref{lemma:Poisson_op} implies that this mapping is injective.  
	It remains to prove surjectivity. Let  $g \in \mathrm{ker}(D_{\ell,\omega} + b)$
	and denote by $ u \in H^1(\Omega)$ the unique solution of the Poisson problem \eqref{eq:Poisson_bvp}. 
	Then $u \in D[a_{\ell,b}]$ and for arbitrary $v \in D[a_{\ell,b}]$  we have
	\begin{align*}
		a_{\ell,b}[u,v] &= \scal{\nabla u}{\nabla v}_{\Omega} + \scal{bg}{v|_{\R \times \{0\}}}_\R =   
		\omega \scal{u}{v}_{\Omega}  + 
			\scal{D_\omega g +b g }{v|_{\R \times \{0\}}}_{\R } \\
		&=  \omega \scal{u}{v}_{\Omega }  + \scal{(D_{\ell,\omega} + b)g }{v|_{\R \times \{0\}}}_{\Sigma_\ell} 
		=\omega \scal{u}{v}_{\Omega} .
	\end{align*}
	Thus,  $u \in D(A_{\ell,b})$ and $(A_{\ell,b} - \omega) u = 0 $. This proves the assertion.
\end{proof}
A particular consequence of Lemma \ref{lemma:kernel} is the observation that the  Dirichlet-to-Robin operator 
$D_{\ell,\omega} + b$ has non-trivial kernel if and 
only if $\omega$ is an eigenvalue of $A_{\ell,b}$. 
Put $V := \tilde H^{1/2}_0(\Sigma_\ell)$ and consider the Gelfand triple
$$ V \to L_2(\Sigma_\ell) \to V^* . $$ 
We identify $V$ with a subspace of $L_2(\Sigma_\ell)$ and $V^* = H^{-1/2}(\Sigma_\ell)$ is the space of 
antilinear functionals on $\tilde H^{1/2}_0(\Sigma_\ell)$, cf. the dual pairing  \eqref{eq:dual_pairing}. 
The truncated Dirichlet-to-Robin operator $D_{\ell, \omega}$ maps 
$$ D_{\ell, \omega}  + b: V \to V^*  , $$
and thus, it is completely described by the sesquilinear form
\begin{align}\label{def:d_ell}
	 (d_{\ell, \omega} + b) [g,h] :=  \scal{ ( D_{\ell,\omega}+ b) g}{h}_{V^*,V} = \scal{D_{\ell,\omega} g}{h}_{\Sigma_\ell} + 
	 \scal{b g}{h}_{\Sigma_\ell}
\end{align}
where $g,h \in D[d_{\ell,\omega}] :=  \tilde H^{1/2}_0(\Sigma_\ell)$. Using the dual 
pairing \eqref{eq:dual_pairing} we obtain
\begin{align}\label{eq:d_ell}
	d_{\ell, \omega}[g,h] 	&= \scal{D_\omega g}{h}_{\R} = \int_\R \sqrt{\xi^2 - \omega} \coth(\alpha \sqrt{\xi^2-\omega}) \cdot \hat g(\xi) \; \overline{\hat h(\xi)}  \ \mathrm{d} \xi 
\end{align}
and 
\begin{align}	b[g,h] &= \int_{\Sigma_\ell} b(x) \cdot g(x) \; \overline{h(x)} \ \mathrm{d} x  .\label{eq:b_ell} 
\end{align}
We note that \eqref{eq:d_ell} is  independent of $\ell$ and the dependence of $\ell$ in \eqref{eq:b_ell} is manifested in
the domain of integration. In particular we may consider the dependence on $\ell$ as a constraint  on the support of
the functions $g$ and $h$. 
\begin{lemma}\label{lemma:form_DtoN}
	Let $\omega \in \C \backslash [\pi^2/\alpha^2, \infty)$. Then $d_{\ell,\omega}$ defines a closed sectorial form 
	in $L_2(\Sigma_\ell)$. The associated m-sectorial operator is the restriction of $D_{\ell,\omega} +b$ to  the operator domain
  	\begin{equation}
		X_{\ell,\omega}  := \left\{ g \in \tilde H^{1/2}_0 (\Sigma_\ell) :  D_{\ell,\omega} g \in L_2(\Sigma_\ell) \right\} .
	\end{equation}
\end{lemma}
\begin{proof}
	Combining  Formulae \eqref{eq:d_ell} and \eqref{eq:b_ell} we have for  $g ,h\in \tilde H_0^{1/2}(\Sigma_\ell) $
	$$ (d_{\ell, \omega} + b )[g,h] =  \int_\R  m_\omega(\xi)  \cdot \hat g(\xi) \; \overline{\hat h(\xi) }
		\ \mathrm{d} \xi + \int_{\Sigma_\ell} b(x) \cdot g(x) \; \overline{h(x)} \ \mathrm{d} x  $$
	with $m_\omega(\xi) = \sqrt{\xi^2 - \omega} \coth(\alpha \sqrt{\xi^2 - \omega})$. 
	Note that 
	$$   m_\omega(\xi) = \sqrt{\xi^2 - \omega} \coth(\alpha \sqrt{\xi^2 - \omega}) = |\xi| + O(1)  \qquad \text{as} \quad \xi \to \pm \infty .       $$
	As $\tilde H_0^{1/2}(\Sigma_\ell)$ carries the subspace topology induced by $H^{1/2}(\R)$ this implies 
	\begin{equation}\label{eq:d_ell_omega_2D_closed}
		 c_1^{-1}  \|g \|_{\tilde H_0^{1/2}(\Sigma_\ell)}^2 \le \Re (d_{\ell, \omega} + b)[g] + c_2 \| g\|_{L_2(\Sigma_\ell)}^2  \le 
	 c_1 \|g \|_{\tilde H_0^{1/2}(\Sigma_\ell)}^2 ,
	\end{equation}
	for constants $c_i = c_i(\omega,\alpha, \Sigma_\ell) \in \R$, $i=1,2$. Thus,  the form $d_{\ell, \omega}$ is bounded
	from below and closed. Moreover, it easily follows that $d_{\ell,\omega} +b$ is sectorial. 
	To prove the second assertion let  $g  \in X_{\ell,\omega}$ such that $D_{\ell,\omega} g = \tilde f\in L_2(\Sigma_\ell)$. Then,
	$$ (d_{\ell, \omega}+b)[g,h] = \scal{D_{\ell,\omega} g}{h}_{\Sigma_\ell} + \scal{bg}{h}_{\Sigma_\ell}
	= \scal{\tilde  f + bg}{h}_{\Sigma_\ell} $$
	for all $h \in C_c^\infty(\Sigma_\ell)$. As  $ C_c^\infty (\Sigma_\ell)$ is dense in $\tilde H_0^{1/2}(\Sigma_\ell)$
	the above equality holds true for every $h \in  \tilde H_0^{1/2}(\Sigma_\ell)$. Thus, $g$ lies
	in the operator domain of the m-sectorial realisation which acts as $D_{\ell, \omega} + b$.
	In the same way it follows that the domain of the m-sectorial realisation is contained in $X_{\ell,\omega}$. 
\end{proof}
Since
$$ \mathrm{ker} ( D_{\ell,\omega} + b) \subseteq X_{\ell,\omega} , $$ 
we can apply Hilbert space methods to determine whether zero is an eigenvalue of $D_{\ell,\omega} +b $ or not. 
The spectrum of the m-sectorial realisation consists of a discrete set of eigenvalues only accumulating at infinity, since 
$D[d_{\ell,\omega} +b] = \tilde H^{1/2}_0(\Sigma_\ell)$ is compactly embedded into $L_2(\Sigma_\ell)$,  cf.\ \cite[Theorem 3.27]{McLean}. 
Moreover, for real $\omega \in \R \backslash [\pi^2/\alpha^2 , \infty)$ the 
quadratic form $d_{\ell, \omega} +b$ is symmetric, and thus, the associated operator is self-adjoint. 
\begin{remark}
	The close relation between self-adjoint extensions of differential operators 
	and self-adjoint operators acting on the boundary has been pointed out in the case of  bounded domains by G.\;Grubb,
	in particular with regard to resolvent formulae, cf.\ \cite{Grubb11} and the references therein. 
\end{remark}
Another consequence of Lemma \ref{lemma:form_DtoN} or merely its proof is the following lemma. 
\begin{lemma} 
	The (original) operator   $D_{\ell, \omega} + b : \tilde H^{1/2}_0(\Sigma_\ell) \to H^{-1/2}(\Sigma_\ell)$ is an  Fredholm operator with zero  index.
\end{lemma}
The  proof follows by combining Formula \eqref{eq:d_ell_omega_2D_closed} and   \cite[Theorems 2.34 and 3.27]{McLean}. 

As the m-sectorial realisation of the Dirichlet-to-Robin operator is simply the restriction of the original operator
we do not want to introduce a separate notation for it. In fact, we will mainly work with a  quadratic form, which arises after  scaling the Robin window.   
Recall that $\Sigma_\ell :=  \ell \cdot \Sigma$. We define the  unitary scaling operator
\begin{align}
	T_\ell  :  L_2(\Sigma)  \to  L_2(\Sigma_\ell), \qquad (T_\ell g)(x) = 
	\ell^{-1/2} g\left(\frac{x}\ell\right) .
\end{align}
Note that the operator $T_\ell$ bijectively maps  $\tilde H^{1/2}_0(\Sigma)$ into $\tilde H^{1/2}_0(\Sigma_\ell)$. Set 
\begin{align}
	\mathcal Q_b(\ell, \omega) : \tilde H^{1/2}_0(\Sigma) \to H^{-1/2}(\Sigma), \qquad \mathcal Q_b(\ell, \omega) := T_\ell^* (D_{\ell, \omega} + b) T_\ell 
\end{align}
and let 
\begin{align}
	q_b(\ell, \omega ) [g,h] :=  ( d_{\ell,\omega} + b)[T_\ell g, T_\ell h ]  
\end{align}
be the associated sesquilinear form with  $D[q_b(\ell, \omega ) ] := \tilde H^{1/2}_{0}(\Sigma)$. 
Then 
$$ \mathrm{dim}\; \mathrm{ker} (A_{\ell,b} - \omega) = \mathrm{dim}\; \mathrm{ker} (D_{\ell,\omega} + b)  = \dim \;  \ker (\mathcal Q_b(\ell, \omega) ) . $$

Next we prove an asymptotic expansion of the operator  $\mathcal Q_{b}(\ell, \omega)$ as  $\ell \to 0$ and $\omega \to \pi^2/\alpha^2$. This expansion represents the principal tool of the proof 
of the main result. Here and subsequently we denote by $\mathcal Q_0 : \tilde H_0^{1/2}(\Sigma) \to H^{-1/2}(\Sigma)$, 
\begin{align}
	\scal{\mathcal Q_0 g}{h}_{\Sigma} := q_0[g,h] := \int_{\R} |\xi| \cdot \hat g(\xi) \; \overline{\hat h (\xi)} \ \mathrm{d} \xi
\end{align}
the Dirichlet-to-Neumann operator for  the mixed problem  on the upper half-space or equivalently on the lower half-space corresponding to the spectral 
parameter $\omega = 0$. Note that   $\mathcal Q_0 : \tilde H^{1/2}_0(\Sigma) \to  H^{-1/2}(\Sigma)$ is also 
a Fredholm operator with Fredholm index $0$, which follows from \cite[Theorem 2.34]{McLean}. The identity $\mathcal Q_0 g = 0$ implies 
$$ 0= \scal{\mathcal Q_0 g}{g}_{\Sigma} = \int_{\R} |\xi| \cdot  | \hat g(\xi)|^2 \ \mathrm{d} \xi , $$
and thus, $g = 0$. Hence  $\mathcal Q_0$ has trivial kernel and it  is invertible. 

In what follows we denote by  $P_\ct$ the projection in $L_2(\Sigma)$ onto the subspace of constant functions and let
$K_{\ln|x|} : L_2(\Sigma) \to L_2(\Sigma)$,  
\begin{equation}
	( K_{\ln|x|} f ) (z) = \int_{\Sigma} \ln |z-x| \cdot f(x) \ \mathrm{d} x , \qquad x \in  \Sigma . 
\end{equation}

\begin{theorem}\label{th:form_asymptotics_2D}
	Let $b = 0$. There exist $\ell_0 = \ell_0(\alpha, \Sigma) > 0$ and $\varepsilon  = \varepsilon(\alpha, \Sigma) > 0 $ such that for 
	$\ell \in ( 0,\ell_0)$ and  $| \omega - \pi^2/\alpha^2| < \varepsilon$   the following asymptotic 
	expansion holds true 
	\begin{align}\label{eq:asympt_q_ell,omega_2D}
		\mathcal Q_0(\ell, \omega) &=  \frac{1}\ell \mathcal Q_0 
			-  \ell \cdot  \left( \frac{|\Sigma| \cdot \pi^2}{\alpha^3} \right) \cdot  \frac{1}{\sqrt{\pi^2/\alpha^2 -\omega }} \;  P_\ct \\
			&\quad +  \sum_{k_1=1}^\infty \sum_{k_2=0}^\infty \ell^{2 k_1 - 1} \left(\sqrt{\pi^2/ \alpha^2 - \omega}\right)^{k_2} 
				\Bigr(B_{k_1,k_2}^{(0)} +  B_{k_1,k_2}^{(1)} \cdot \ln \ell \Bigr) .  
	\end{align}
	Here  $B_{k_1,k_2}^{(i)}   \in \mathcal   L(L_2(\Sigma))$ for $i=1,2$. The series converges absolutely in the operator norm of  $\mathcal L(L_2(\Sigma))$. 
	For the first terms we obtain
	\begin{align*}  
		B_{1,0}^{(0)} &=  \frac{|\Sigma| \cdot \rho(\alpha)}{2\pi} P_\ct +   \frac{\pi }{2 \alpha^2} K_{\ln|x|}  , \qquad B_{1,0}^{(1)} =  \frac{|\Sigma| \cdot \pi}{2 \alpha^2} P_\ct , 
		\qquad B_{1,1}^{(0)} =  B_{1,1}^{(1)} = 0 ,
	\end{align*}
	where the constant  $ \rho(\alpha)  \in \R$ is given by Formula \eqref{eq:rho_2D} and $|\Sigma|$ is the Lebesgue measure of $\Sigma$.  
\end{theorem}
The next section is devoted to the proof of Theorem \ref{th:form_asymptotics_2D}. 

\subsection*{The proof of Theorem \ref{th:form_asymptotics_2D}} \label{subsec:asymptotics_2D}
For $g, h \in \tilde H^{1/2}_0(\Sigma)$ we have 
$$ \scal{\mathcal Q_{0}(\ell, \omega) g}{h}_\Sigma = q_0(\ell,\omega) [g,h ] = \ell  \int_{\R}  m_\omega(\xi) \cdot \hat g(\ell \xi) \; \overline{h(\ell \xi)}  \ \mathrm{d}  \xi, $$
where   $m_\omega(\xi) = \sqrt{\xi^2 - \omega} \cdot \coth(\alpha \sqrt{\xi^2 - \omega})$. 
The main idea of the proof is to  
use the asymptotic expansion of the function $m_\omega(\xi) $ for $\xi =  0$ and $\xi \to \pm \infty$
while letting the parameter $\omega \to \pi^2/\alpha^2$.  

As a first step of the proof we show that $m_\omega$ has a meromorphic extension to the
complex plane and calculate explicitly its singularities and residues. 
To  this end we use the  partial fraction decomposition of the hyperbolic
cotangent function, i.e., we have 
\begin{equation}\label{eq:series_coth}
	\coth( z) = \frac1z +   \sum_{k=1}^\infty \frac{2z}{z^2+k^2 \pi^2}  ,\qquad  z \in \C \backslash 
	\{ \mathrm{i} k \pi : k \in \Z\},
\end{equation} 
cf.\ e.g.\  \cite[Chapter V, \S 1.71]{LawSchabat}. 
Hence,  
\begin{equation}\label{eq:series_m_omega} 
	m_\omega(\xi ) =
	\frac1{\alpha} + \sum_{k=1}^\infty \frac{2 \alpha (\xi^2 - \omega)}{\alpha^2 (\xi^2 - \omega) +k^2 \pi^2 }  , \qquad
	\xi \in \R,
\end{equation}
and the  meromorphic extension of $m_\omega$ is given by the series \eqref{eq:series_m_omega}. For $\omega \in \C \backslash [\pi^2/\alpha^2 , \infty) $ the singularities of $m_\omega$ are all simple poles which are located at 
\begin{equation}
	\pm \mathrm{i} \sqrt{\frac{k^2 \pi^2}{\alpha^2}-\omega} , \qquad k \in \N . 
\end{equation}
In particular they do  not lie on the real axis. 
As $\omega \to \pi^2/\alpha^2$ the two poles nearest the real axis converge to  $0 \in \C$ and they 
give rise to a pole of order two in the limit case. 
Here and subsequently we fix $\beta = \beta(\alpha) \in (\pi/\alpha, \sqrt{3} \pi/\alpha)$. Then there exists $\varepsilon = \varepsilon (\alpha) > 0$  such that 
for $\omega \in \C \backslash [\pi^2/\alpha^2, \infty)$, $|\omega  - \pi^2/\alpha^2| < \varepsilon$ the function $m_\omega$ has exactly two poles 
inside the strip $\R + \mathrm{i} [- \beta, \beta]$. 
The residues of the function $m_\omega$ at these  points are given by 
%are given by 
\begin{equation}\label{eq:residue_m_omega}
	\mathrm{Res}_{\xi = \pm \mathrm{i} \sqrt{\pi^2/\alpha^2-\omega} } \; m_\omega(\xi) = \pm \frac{\pi^2}{\alpha^3} \cdot \frac{\mathrm{i} }{\sqrt{\pi^2/\alpha^2-\omega}} 
\end{equation}
as can easily be seen from the expansion \eqref{eq:series_m_omega}. 

Let $g ,h \in \tilde H^{1/2}_0(\Sigma)$. 
Since $g,h$ are compactly supported their Fourier transforms $\hat g, \hat h$ admit holomorphic extensions
on the whole complex plane. Note that the function $\hat h^*$, 
$$ \hat h^*(\xi) := \overline{\hat h (\overline \xi)}, \qquad \xi \in \C, $$
is also an entire function on $\C$. We decompose the form $q_0(\ell, \omega)$ as follows 
\begin{align}
	q_0(\ell,\omega)[g,h ] 
		&= \ell \left( \int_{[-1,1]} + \int_{[-1,1]^c} \right) m_\omega(\xi) \cdot \hat g(\ell\xi) \; \hat h^*(\ell \xi)  \ \mathrm{d} \xi  ,	
		\label{eq:decomp_q_ell_2D}
\end{align}
where we put $[-1,1]^c = \R \backslash [-1,1]$. 
Using the Taylor expansion of the function $\hat g \cdot \hat h^*$ at $0 \in \C$, we obtain for the first integral 
\begin{align}\label{eq:asympt_q_ell_1}
	\int_{[-1,1]}  m_\omega(\xi) \cdot \hat g(\ell \xi) \; \hat h^*(\ell \xi)  \ \mathrm{d} \xi  
	&= \sum_{k=0}^\infty \ell^k  e_k[g,h]    \int_{[-1,1]} \xi^k m_\omega(\xi) \ \mathrm{d} \xi 
\end{align}
with 
\begin{equation}
	e_{k}[g,h] = \frac1{k!} \cdot \left. \frac{\mathrm{d}^{k}}{\mathrm{d} \xi^{k}} ( \hat g (\xi)  \hat h^*(\xi) ) \right|_{\xi =0 } 
	= \frac1{k!} \cdot  \sum_{j=0}^{k} \binom{k}j  \hat g^{(j)} (0) \cdot \overline{\hat h^{(k-j)} (0)} .  
\end{equation}
We note that  $m_\omega$ is an even function, and thus, in the expansion the terms of odd order vanish. Let 
$E_{k}$ be the operator associated with the form $e_{k}$. Then 
$$ \int_{[-1,1]}  m_\omega(\xi) \cdot \hat g(\ell \xi) \; \hat h^*(\ell \xi)   \ \mathrm{d} \xi  
	= \sum_{k=0}^\infty \ell^{2k}  \scal{E_{2k} g}{h}_{\Sigma}     \int_{[-1,1]} \xi^{2k} m_\omega(\xi) \ \mathrm{d}  \xi . $$
Note that
$$ | \hat g^{(j)} (0) | \le \frac{1}{\sqrt{2\pi}} \left ( \int_{\Sigma}  |x|^{2j} \right)^{1/2}  \| g\|_{L_2(\Sigma)} 
	\le C^j \| g\|_{L_2(\Sigma)} $$
for sufficiently large $C = C(\Sigma) > 0$, which implies
$$ \|E_{k} \|_{\mathcal L(L_2(\Sigma))} \le \frac{(2C)^{k }}{k!}  . $$
To estimate the integral $\int_{[-1,1]} \xi^{2k} m_\omega(\xi) \ \mathrm{d} \xi $ we denote by  $\gamma$
the path depicted in Figure \ref{fig:rect_2D} connecting 
the points $-1$ and $1$. Its  image $\mathrm{im}(\gamma)$ coincides with the boundary of the following rectangle except for the line segment $[-1,1]$. 
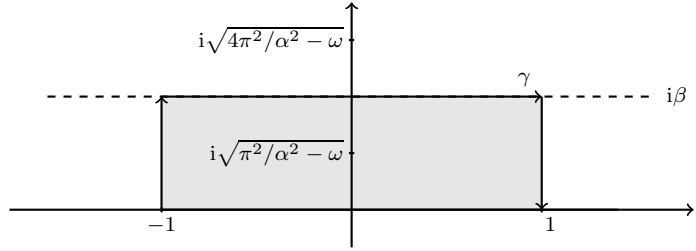
\begin{figure}[H]
\centering
\begin{tikzpicture}
 \begin{scope}[thick,font=\scriptsize]
  	\draw[->]  (-4.5,0) -- (4.5,0);
  	\draw[->] (0,-0.5) -- (0,2.75);
 	\draw (-2.5,-1pt) -- (-2.5,1pt) node[anchor=north] {$-1$}; 
  	\draw (2.5,-1pt) -- (2.5,1pt) node[anchor=north] {$\phantom{-}1$}; 
 	\draw (-2.5,0) -- (3.5,0);
 	\draw [<-] (2.5,0) -- (2.5,1.5);
  	\draw [->] (-2.5,1.5) -- (2.5,1.5) node[anchor=south east] {$\gamma$};
  	\draw [->] (-2.5,0) -- (-2.5,1.5);
	\fill [opacity=0.1] (-2.5,0) rectangle (2.5,1.5); 
  	\draw [dashed] (-4,1.5) -- (4,1.5) node[anchor=west] {$\mathrm{i} \beta$};  	
  	\draw (-1pt,0.75) -- (1pt,0.75) node[anchor=east] {$\mathrm{i} \sqrt{\pi^2/\alpha^2-\omega}$};
	\draw (-1pt,2.25) -- (1pt,2.25) node[anchor=east] {$\mathrm{i} \sqrt{4 \pi^2/\alpha^2-\omega}$};
\end{scope}
\end{tikzpicture}
\caption{The path $\gamma$.}
\label{fig:rect_2D}
\end{figure}
Note that $\Im(\mathrm{i} \sqrt{\pi^2/\alpha^2-\omega}) > 0$ if $\omega \in \C \backslash[\pi^2/\alpha^2, \infty)$.
Using Formula \eqref{eq:residue_m_omega}  the residue theorem implies 
\begin{equation}
	\int_{[-1,1]} \xi^{2k} m_{\omega} (\xi) \ \mathrm{d} \xi = 
	 - \frac{(-1)^k \cdot  2 \pi^3}{\alpha^3} (\pi^2 / \alpha^2 - \omega)^{k-1/2} + \int_{\gamma} \xi^{2k} m_\omega (\xi) \ \mathrm{d} \xi  . 
\end{equation}
Next we use for fixed $\xi \in \mathrm{im}(\gamma)$  the Taylor expansion of 
$m_\omega(\xi)$ at $\omega = \pi^2 /\alpha^2$. Thus, there exists $\varepsilon >0$ such that for $|\omega - \pi^2/\alpha^2| < \varepsilon$ we have 
\begin{equation} 
	m_\omega(\xi) = \sum_{j=0}^\infty \frac{(-1)^j  (\pi^2/\alpha^2 - \omega)^j }{j!}  \left. \frac{\mathrm{d}^j}{\mathrm{d} \omega^j} m_\omega (\xi) \right|_{\omega = \pi^2/\alpha^2} . 
\end{equation}
This expression may be considered as a power series in $\omega$ with values in $L_1(\mathrm{im} (\gamma))$, and we obtain %, since
$$ \int_{\gamma} \xi^{2k} m_\omega (\xi)  \ \mathrm{d} \xi  = \sum_{j=0}^\infty \frac{(-1)^j   (\pi^2/\alpha^2 - \omega)^j }{j!} 
	\left[  \int_{\gamma} \xi^{2k} \frac{\mathrm{d}^j}{\mathrm{d} \omega^j} m_\omega (\xi) \ \mathrm{d} \xi \right]_{\omega = \pi^2/\alpha^2} .  $$
Finally,
\begin{align*}
	&\ell \int_{[-1,1]} m_\omega(\xi) \cdot \hat g(\ell \xi) \; \hat h^*(\ell \xi)   \ \mathrm{d} \xi \\
	&\quad = - \ell \left(\frac{ 2  \pi^3}{\alpha^3} \right) \cdot \frac{1}{\sqrt{\pi^2/\alpha^2  - \omega}} \scal{E_0 g}{h}_\Sigma \\
	&\qquad - \left( \frac{2 \pi^3}{\alpha^3} \right) \cdot \sum_{k = 1}^\infty \ell^{2k+1} \cdot (-1)^k \cdot (\pi^2 / \alpha^2 - \omega)^{k-1/2} \cdot \scal{E_{2k} g}{h}_{\Sigma} \\
	&\qquad  + \sum_{k=0}^\infty \sum_{j=0}^\infty \ell^{2k+1}\scal{E_{2k} g}{h}_{\Sigma} 
		 \frac{(-1)^j (\pi^2/\alpha^2 - \omega)^j }{j!} 
	\left[  \int_{\gamma} \xi^{2k} \frac{\mathrm{d}^j}{\mathrm{d} \omega^j} m_\omega (\xi) \ \mathrm{d} \xi \right]_{\omega = \pi^2/\alpha^2} 
		\hspace{-0.5cm}.
\end{align*}
We note that the two series 
\begin{align*}
	\sum_{k = 1}^\infty \ell^{2k+1} E_{2k} \cdot (-1)^k \cdot  (\pi^2 / \alpha^2 - \omega)^{k-1/2} 
\end{align*}
and 
\begin{align*}
	\sum_{k=0}^\infty\sum_{j=0}^\infty  \ell^{2k+1} E_{2k} 
		 \frac{(-1)^j (\pi^2/\alpha^2 - \omega)^j }{j!} 
	\left[  \int_{\gamma} \xi^{2k} \frac{\mathrm{d}^j}{\mathrm{d} \omega^j} m_\omega (\xi) \ \mathrm{d} \xi \right]_{\omega = \pi^2/\alpha^2}
\end{align*}
converge absolutely in the operator norm in $\mathcal L(L_2(\Sigma))$,
uniformly in $\ell \in [0,1]$ and $|\omega - \pi^2/\alpha^2| < \varepsilon$.
For the first series this is obvious. Considering  the second series leads us to the estimate 
\begin{align*}
	&\sum_{k=0}^\infty \ell^{2k+1} \| E_{2k} \|_{\mathcal L(L_2(\Sigma))} 
		\sum_{j=0}^\infty \frac{|\pi^2/\alpha^2 - \omega|^j }{j!} 
	\left|  \int_{\gamma} \xi^{2k} \frac{\mathrm{d}^j}{\mathrm{d} \omega^j} m_{\pi^2/\alpha^2} (\xi) \ \mathrm{d} \xi \right| \\
	&\qquad \le \left( \sum_{k=0}^\infty \ell^{2k+1} \cdot c_1^{2k}    \cdot \frac{(2C)^{2k}}{(2k)!}  \right) 
	\left( \sum_{j=0}^\infty \frac{\varepsilon^j}{j!}   
		\left\| \frac{\mathrm{d}^j}{\mathrm{d} \omega^j} m_{\pi^2/\alpha^2} \right\|_{L_1(\mathrm{im} (\gamma))} \right) < \infty  
\end{align*}
for $\ell > 0$ and $|\omega - \pi^2/\alpha^2| < \varepsilon$. Here $c_1 :=  \sup \{  |\xi| : \xi \in \mathrm{im}(\gamma) \}$. 
Calculating the first terms of the expansion we obtain  
\begin{align*}
	& \ell \int_{[-1,1]}  m _\omega (\xi) \cdot \hat g (\ell \xi) \; \hat h^* (\ell \xi) \ \mathrm{d} \xi  \\ & \quad = 
	 - \ell \cdot \frac{2 \pi^3  \bigr(\hat g \hat h^*\bigr)(0)  }{\alpha^3  \sqrt{\pi^2/\alpha^2 - \omega }} 
	 + \ell \cdot \bigr(\hat g \hat h^*\bigr)(0)  \int_{\gamma} m_{\pi^2/\alpha^2}(\xi) \ \mathrm{d} \xi %\\[6pt]
	 + \mathcal O(\ell^3 + \pi^2/\alpha^2 - \omega) . 
\end{align*}
Let  $P_\ct$ be the projection in $L_2(\Sigma)$  onto the subspace of constant functions. Then 
\begin{align}
	\bigr(\hat g \hat h^*\bigr)(0)  = \frac1{2\pi} \left( \int_{\Sigma} g(x) \ \mathrm{d} x \right) \left( \int_{\Sigma} \overline{h(x)} \ \mathrm{d} x \right)  = \frac{|\Sigma|}{2\pi}  \scal{P_\ct g}{h}_{\Sigma} ,
\end{align}
and thus, 
\begin{align*}
	&\ell \int_{[-1,1]} m_\omega(\xi) \cdot \hat g (\ell \xi) \; \hat h^* (\ell \xi)   \ \mathrm{d} \xi\\
	&\quad =  \ell \biggr(-  \frac{|\Sigma| \cdot  \pi^2}{\alpha^3 \sqrt{\pi^2/\alpha^2 - \omega}}
	+ \frac{|\Sigma| \cdot \rho_{0,1}(\alpha)}{2\pi} \biggr)  \scal{P_\ct g}{h}_{\Sigma}  
	+  \mathcal O(\ell^3 + \pi^2/\alpha^2 - \omega) , 
\end{align*}
with $ \rho_{0,1}(\alpha) = \int_{\delta} m_{\pi^2/\alpha^2} (\xi) \ \mathrm{d} \xi$. 
In order to treat the second  integral in \eqref{eq:decomp_q_ell_2D} we
use the asymptotic expansion of $m_\omega(\xi)$ for $\xi \to \pm \infty$. 
For ease of notation we suppose that  $\alpha >   \pi$, so  that $[-1,1] \subseteq ( -\pi/\alpha, \pi/\alpha )$. We have 
\begin{align*}
	&\ell \int_{[-1,1]^c} m_\omega(\xi) \cdot \hat g(\ell \xi) \; \hat h^*(\ell \xi) \ \mathrm{d} \xi \\ &\qquad = \ell \int_{\R} 
		|\xi| \cdot \hat g(\ell \xi) \; \hat h^*(\ell \xi) \ \mathrm{d} \xi  +  \ell \sum_{i=1}^3 \int_\R  m_{\omega,i} (\xi) \cdot \hat g(\ell \xi) \; \hat h^*(\ell \xi)  \ \mathrm{d} \xi , 
\end{align*}
where 
\begin{align}
	m_{\omega,1}(\xi)  &:=  - \mathds 1_{[-1,1]} (\xi) \cdot |\xi| , \\[5pt]
	m_{\omega,2} (\xi) &:= \mathds{1}_{[-1,1]^c}  (\xi) \cdot \Bigr( \sqrt{\xi^2 - \omega} - |\xi| \Bigr) , \\[5pt] 
	m_{\omega,3} (\xi) &:= 
	 \mathds{1}_{[-1,1]^c}(\xi)  \sqrt{\xi^2 - \omega} \left( \coth(\alpha \sqrt{\xi^2  - \omega}) -1 \right) .
\end{align}
Here $ \mathds{1}_E$ denotes the indicator function of a Borel set $E\subseteq \R$. 
We note that the first term $\int_\R  m_{\omega,1} (\xi) \cdot \hat g(\ell \xi) \; \hat h^*(\ell \xi)  \ \mathrm{d} \xi $ is independent of $\omega$ and may be expanded as before into a power series with respect to 
the parameter $\ell$; we have 
\begin{align*}
	\ell \int_\R m_{\omega,1} (\xi) \cdot \hat g(\ell \xi) \; \hat h^*(\ell \xi)   \ \mathrm{d} \xi &= \ell \int_{[-1,1]} |\xi| \cdot \hat g(\ell \xi) \; \hat h^*(\ell \xi)   \ \mathrm{d} \xi  \\
	&=  \frac{\ell |\Sigma| }{2 \pi} \scal{P_\ct g}{h}_{\Sigma} \cdot \int_{[-1,1]} |\xi| \ \mathrm{d} \xi + \mathcal O(\ell^3)  \\
	&= \frac{\ell |\Sigma| }{2 \pi} \scal{P_\ct g}{h} + \mathcal O(\ell^3) .
\end{align*}
To treat the second integral we use for   $\xi \in \R $ and $|\omega - \pi^2/\alpha^2|  < \varepsilon $  the following  expansion
\begin{align*}\label{eq:exp_m_omega_symbol}
	m_{\omega,2} (\xi) &= \mathds{1}_{[-1,1]^c}(\xi) \left( \sqrt{\xi^2 - \omega} - |\xi| \right) 
	= \mathds{1}_{[-1,1]^c}(\xi) \cdot \left(
	|\xi| \cdot \sqrt{ 1- \frac{\omega}{\xi^2}} - |\xi| \right) \\
	&= \sum_{k=1}^\infty \binom{1/2}{k}   (- \omega) ^k \cdot | \xi|^{-2k+1} \mathds{1}_{[-1,1]^c}(\xi)  .
\end{align*} 
This series may be considered as power series  in $ \omega $ with values in $L_\infty(\R_\xi)$. 
Let $Y_k , Z_k \in C^\infty(\R \backslash\{0\}) \cap L_{1,\mathrm{loc}}(\R)$ such that 
\begin{align*}
	\hat Y_k (\xi ) = \frac{1}{\sqrt{2\pi}} \cdot | \xi|^{-2k+1} \mathds{1}_{[-1,1]^c}(\xi) \qquad \text{and} \qquad 
	\hat Z_k (\xi) =  \frac{1}{\sqrt{2\pi}}  \; \mathrm{f.p.}\left(|\xi|^{-2k+1}\right) . 
\end{align*}
Here $$ \mathrm{f.p.}\left(|\xi|^{-2k+1}\right)  = \mathrm{f.p.} \left(\xi_+^{-2k+1}\right) +  \mathrm{f.p.} \left(\xi_-^{-2k+1}\right) $$ 
designates the distribution which is defined by the finite part of the  singular function $|\xi|^{-2k+1}$, cf.\
\cite[Chapter 5]{McLean}. We note  that $X_k := Y_k - Z_k$ is analytic, since its Fourier transform $\hat X_k = \hat Y_k - \hat Z_k$ has
 compact support. This allows us to determine the order of the  singularity of $Y_k$ at $0\in \R$. 
Using \cite[Lemma 5.10]{McLean}
we have 
\begin{align}
	Z_k (x) = \frac{1}{\pi} \cdot  \frac{ (-1)^{k} x^{2k-2}}{(2k-2)!} \left( \ln |x| + \gamma_0 - H_{2k-2}
	\right) ,
\end{align}
where $\gamma_0$ is the Euler-Mascheroni constant and 
$H_{2k-2}  = \sum_{j=1}^{2k-2} \frac{1}{j}$. We have 
%we obtain 
\begin{align*}
	\ell \int_{\R} m_{\omega,2} (\xi) \cdot \hat g(\ell \xi) \; \hat h^*(\ell \xi) \ \mathrm{d} \xi 
	&= \sum_{k=1}^\infty \binom{1/2}{k} (- \omega)^k \cdot \ell \cdot \int_{\R} \hat Y_k(\xi) \cdot \hat g(\ell \xi) \; \hat h^*(\ell \xi) \ \mathrm{d} \xi \\
	&= \sum_{k=1}^\infty \binom{1/2}{k}  (- \omega)^k \cdot \scal{Y_k \ast T_\ell g}{T_\ell h}_{\Sigma_\ell} , \\[5pt]
	\scal{Y_k \ast T_\ell g}{T_\ell h}_{\Sigma_\ell}  
	&= \ell   \int_{\Sigma \times \Sigma} (X_k  + Z_k) (\ell |z-x|) \; g(x)\; \overline{h (z)} \ \mathrm{d} x \ \mathrm{d} z . 
\end{align*}
Defining the operators  $K_{|x|^{2k-2}} , K_{|x|^{2k-2}\ln |x|} : L_2(\Sigma) \to L_2(\Sigma)$, 
\begin{align}
	(K_{|x|^{2k-2}} f)(z) &= \int_{\Sigma} | z - x |^{2k-2} \; f(x) \ \mathrm{d} x ,  \qquad z \in \Sigma \\
	(K_{|x|^{2k-2} \ln |x| } f)(z) &= \int_{\Sigma} | z - x|^{2k-2} \ln |z -x| \; f(x) \ \mathrm{d} x , \qquad z\in \Sigma
\end{align}
we obtain
\begin{align*}
	& \ell  \int_{\Sigma \times \Sigma}  Z_k (\ell |z - x|) \; g(x) \;\overline{h  (z)} \ \mathrm{d} x \ \mathrm{d} z \\
	&\quad = \frac{(-1)^{k} \; \ell^{2k-1}}{ \pi \cdot (2k-2)!} 
		  \int_{\Sigma \times \Sigma}  |z- x|^{2k-2} \ln (\ell |z-x|) \; g(x)\; \overline{h (z)} \ \mathrm{d} x \ \mathrm{d} z \\
	&\qquad + \frac{(-1)^{k} \; \ell^{2k-1}}{ \pi \cdot (2k-2)!}  \cdot \left( \gamma_0 - H_{2k-2}\right) 
		\int_{\Sigma \times \Sigma}  |z - x|^{2k-2} \;  g(x) \;  \overline{h (z)} \ \mathrm{d} x \ \mathrm{d} z \\[8pt]
	&\quad = \frac{(-1)^{k} \; \ell^{2k-1}}{ \pi \cdot (2k-2)!}   \left( \ln \ell + \gamma_0 - H_{2k-2}  \right) 
	\scal{K_{|x|^{2k-2}} g}{h}_{\Sigma} \\
	&\qquad + \frac{(-1)^{k} \; \ell^{2k-1}}{ \pi \cdot (2k-2)!} 
	\scal{K_{|x|^{2k-2} \ln |x| } g}{h}_{\Sigma}.
\end{align*}
Note that 
\begin{align*}
	& \ell  \int_{\Sigma \times \Sigma}  X_k (\ell |z -x|) \;  g(x) \; \overline{h (z)}  \ \mathrm{d} x \ \mathrm{d} z
	= \sum_{j=0}^\infty \frac{X_k^{(2j)}(0) \cdot \ell^{2j+1} }{(2j)!} \scal{K_{|x|^{2j}} g}{h}_{\Sigma} ,
\end{align*}
which  follows by expanding the even function $X_k$ into a power series. For the coefficients $X_k^{(2j)}(0)$ we obtain from the definition of the finite part
\begin{align*}
	2\pi  X_k^{(2j)}(0) &= 2  \int_0^\infty  (\mathrm{i} \xi)^{2j} (\hat Y_k - \hat Z_k)(\xi) \ \mathrm{d} \xi 
	= - 2 (-1)^j  \mathrm{f.p.}_{\varepsilon \to 0} \int_{\varepsilon}^{1} \xi^{2j} \cdot \xi^{-2k+1}   \ \mathrm{d} \xi \\
	&= (-1)^{j+1} \cdot  \begin{cases} \displaystyle  \frac{1}{j-k + 1  } , & 
		j - k +1  \neq 0, \\[10pt] \displaystyle 0  , & j - k +1  =  0 . \end{cases}
\end{align*}
Finally, we have  
\begin{align*} 
	&\ell \int_{\R} m_{\omega,2} (\xi) \cdot \hat g(\ell \xi) \; \hat h^*(\ell \xi)  \ \mathrm{d} \xi \\
	& = \sum_{k=1}^\infty \binom{1/2}{k} 
	\frac{\omega^k \cdot  \ell^{2k-1}}{\pi (2k-2)!} \Bigr[  \left( \ln \ell + \gamma_0 - H_{2k-2}  \right) \scal{K_{|x|^{2k-2}} g}{h}_{\Sigma}   + \scal{K_{|x|^{2k-2} \ln |x| } g}{h}_{\Sigma} \Bigr] \\
	&\quad +   \sum_{k=1}^\infty \sum_{j=0}^\infty  \binom{1/2}{k}  \cdot (-  \omega)^k \cdot 
		\frac{X_k^{(2j)}(0) \cdot \ell^{2j+1} }{(2j)!} \scal{K_{|x|^{2j}} g}{h}_{\Sigma} .
\end{align*}
Note that both series, 
\begin{equation*}
	\frac{1}{\pi} \sum_{k=1}^\infty \binom{1/2}{k} \cdot \omega^k \cdot 
	\frac{\ell^{2k-1}}{ (2k-2)!} \biggr[  \left( \ln \ell + \gamma_0 - H_{2k-2}  \right) K_{|x|^{2k-2}}
	+ K_{|x|^{2k-2} \ln |x| }  \biggr]
\end{equation*}
and
\begin{equation*}
	 \sum_{k=1}^\infty \sum_{j=0}^\infty  \binom{1/2}{k} (-   \omega)^k \cdot 
		\frac{X_k^{(2j)}(0) \cdot \ell^{2j+1} }{(2j)!} K_{|x|^{2j}} ,
\end{equation*}
converge  uniformly in  the operator norm for $\ell \in (0,\ell_0)$ and $| \omega - \pi^2/ \alpha^2| < \varepsilon$.
This follows from the estimates on the coefficients $X_k^{2j}(0)$ and from 
$$ \|K_{|x|^{2k-2}}\|_{\mathcal L(L_2(\Sigma))} \le C^{2k-2}  , \qquad \| K_{|x|^{2k-2} \ln |x| }\|_{\mathcal L(L_2(\Sigma))}
\le C^{2k-2}   $$ 
for sufficiently large $C = C(\Sigma) >0$.
Note that $\omega  < 1$ since we assumed that  $\pi^2/\alpha^2 < 1$.   
Changing the centre of the power series in $\omega$  and calculating the first terms 
give  us the following asymptotic estimate
\begin{align*}
	&\ell \int_{\R} m_{\omega,2} (\xi) \cdot  \hat g(\ell \xi) \; \hat h^*(\ell \xi) \ \mathrm{d} \xi \\
	&=  \frac{\ell \ln \ell \cdot |\Sigma| \cdot \pi}{2 \alpha^2} \scal{P_\ct g}{h}_{\Sigma} 
	+ \ell \cdot |\Sigma| \cdot  \left (  \frac{\pi\cdot \gamma_0  }{2\alpha^2} + \frac{\rho_{0,2}(\alpha)}{2\pi} \right) \scal{P_\ct g}{h}_{\Sigma} \\[4pt]
	&\quad + \ell \cdot   \frac{\pi}{2\alpha^2} \scal{K_{\ln|x|}g}{h}_{\Sigma} + \mathcal O (\ell^3 \ln \ell + \pi^2 /\alpha^2 - \omega) ,
\end{align*}
where  
\begin{align}
	\rho_{0,2}(\alpha) := \pi  \sum_{k=1}^\infty \binom{1/2}{k} \left( - \frac{\pi^2}{\alpha^2} \right)^k 
	X_k(0) = \frac1{2} \sum_{k=2}^\infty \binom{1/2}{k} \left(- \frac{\pi^2}{\alpha^2} \right)^k  
	\frac{1}{k-1} .
\end{align}
We note that $K_{|x|^0}  = |\Sigma| \cdot P_\ct$. Thus, the only point remaining is the expansion of the integral
$ \ell^2 \int_{\R} m_{\omega,3}(\xi) \cdot \hat g(\ell \xi) \; \hat h^*(\ell \xi) \ \mathrm{d} \xi$.  We have 
\begin{align*}
	m_{\omega,3} (\xi)   
	&=  \mathds{1}_{[-1,1]^c}(\xi) \; \sqrt{\xi^2 - \omega} \;  \left(\coth(\alpha\sqrt{\xi^2 - \omega} ) - 1 \right) . 
\end{align*}
It easily follows that the function 
$$ \omega \mapsto  \mathrm{e}^{\delta  |\xi|}    m_{\omega,3}  (\xi) \in L_\infty(\R_\xi) $$
is a vector-valued holomorphic function for $|\omega - \pi^2/\alpha^2| < \varepsilon$ and some $\delta > 0$. 
In particular, we obtain that 
$$  m_{\omega,3} (\xi) =  \sum_{k=0}^\infty \frac{(-1)^k ( \pi^2/\alpha^2 - \omega)^k}{k!}
	 \left. \frac{\mathrm{d}^k}{\mathrm{d} \omega^k}  m_{\omega,3} (\xi)\right|_{\omega= \pi^2/\alpha^2}    $$
and the  series converges absolutely as  a power series in $\omega$ with values in some exponentially weighted $L_\infty$-space. 
Choose  $\widetilde X_k \in C^\infty(\R)$ such that
 $$\widehat{\widetilde X_k} (\xi) = \frac{1}{\sqrt{2\pi}} \cdot 
\left. \frac{\mathrm{d}^k}{\mathrm{d} \omega^k}  m_{\omega,3} (\xi) \right|_{\omega = \pi^2/\alpha^2} . $$
Then  $\widetilde X$ is an even function and analytic in some neighbourhood of $0$, 
\begin{align*}
	\widetilde X_k (x) &=  \sum_{j=0}^\infty \frac{\widetilde X_k^{(2j)}(0) }{(2j)!} x^{2j} ,
\end{align*}
where 
$\widetilde X_k^{(2j)}(0)  = 
	\frac{1}{2\pi} \int_{\R} (\mathrm{i} \xi)^{2j} \frac{\mathrm{d}^k}{\mathrm{d} \omega^k} m_{\pi^2/\alpha^2,3}(\xi) \ \mathrm{d}  \xi$.
Note that 
\begin{align*}
	|\tilde X_k^{(2j)}(0)| &\le \frac{1}{2\pi} \left\| \mathrm{e}^{\delta |\xi|} 
		\frac{\mathrm{d}^k}{\mathrm{d} \omega^k}  m_{\pi^2/\alpha^2, 3} (\xi)\right\|_{L_\infty(\R_\xi)}   
		2 \int_{0}^\infty  \xi^{2j} \mathrm{e}^{-\delta  \xi} \ \mathrm{d} \xi  \\
		&=  \frac{\delta ^{-1-2j} (2j)!}{\pi} \left\| \mathrm{e}^{\delta  |\xi|} 
		\frac{\mathrm{d}^k}{\mathrm{d} \omega^k}  m_{\pi^2/\alpha^2,3}  (\xi)\right\|_{L_\infty(\R_\xi)}   ,
\end{align*}
since $\int_{0}^\infty  \xi^{2j} \mathrm{e}^{-\delta  \xi} \ \mathrm{d} \xi = \delta ^{-1-2j} (2j)!$,  cf.\ \cite[Formula 2.321]{GradRyz}. In particular, we obtain 
\begin{align*}
	&\ell \int_{\R} m_{\omega,3} (\xi) \cdot \hat g(\ell \xi) \; \hat h^*(\ell \xi)   \ \mathrm{d} \xi \\
	 &\quad= \sum_{k=0}^\infty \frac{(-1)^k  \left(\pi^2/\alpha^2 - \omega\right)^k }{k!}  \scal{\widetilde X_k \ast T_\ell g}{T_\ell h}_{\Sigma_\ell} \\
	&\quad= \sum_{k=0}^\infty \sum_{j=0}^\infty\frac{(-1)^k\left(\pi^2/\alpha^2 - \omega\right)^k }{k!} 
		\frac{\widetilde X_k^{(2j)} (0)}{(2j)!}  \cdot \ell^{2j+1} \cdot 
		\scal{K_{|x|^{2j}} g}{h}_{\Sigma} .  		
\end{align*}
Note that the estimates on the coefficients $\widetilde X_k^{(2j)}(0)$ imply that the series
$$ \sum_{k=0}^\infty \sum_{j=0}^\infty \frac{(-1)^k\left(\pi^2/\alpha^2 - \omega\right)^k }{k!}  
			 \frac{\widetilde X_k^{(2j)} (0)}{(2j)!}  \cdot \ell^{2j+1} \cdot 
			K_{|x|^{2j}} ,  $$
converges in the  operator norm of $\mathcal L(L_2(\Sigma))$, uniformly in $\ell \in [0, \ell_0]$ and
$| \omega - \pi^2 /\alpha^2 | < \varepsilon $.  
In particular, we obtain 
\begin{align*}
	\ell \int_{\R} & m_{\omega,3} (\xi) \cdot \hat g(\ell \xi) \; \hat h^*(\ell \xi)  \ \mathrm{d} \xi \\
	&= \frac{\ell \cdot |\Sigma|\cdot  \scal{P_\ct g}{h}_{\Sigma}}{2\pi} 
		\int_\R m_{\pi^2/\alpha^2,3} (\xi) \ \mathrm{d} \xi + \mathcal O(\ell^3 + \pi^2/\alpha^2  - \omega ) \\
		&=  \frac{\ell \cdot |\Sigma| \cdot \rho_{0,3}(\alpha) }{2\pi}  \scal{P_\ct g}{h}_{\Sigma}
		  + \mathcal O(\ell^3 + \pi^2/\alpha^2  - \omega ) ,
\end{align*}
where $\rho_{0,3}(\alpha) = \int_{\R} m_{\pi^2/\alpha^2,3} (\xi) \ \mathrm{d} \xi$. 
Putting 
\begin{align} \label{eq:rho_2D}
	\rho_0(\alpha) = \rho_{0,1}(\alpha)  + \rho_{0,2}(\alpha) + \rho_{0,3}(\alpha)
	+ 1 + \frac{\gamma_0 \pi^2}{\alpha^2}   ,
\end{align}
with 
\begin{align*}
	\rho_{0,1}(\alpha) &= \int_{\delta} m_{\pi^2/\alpha^2}(\xi) \ \mathrm{d} \xi , &  
	\rho_{0,2}(\alpha) &= \frac1{2} \sum_{k=2}^\infty \binom{1/2}{k} \left(- \frac{\pi^2}{\alpha^2} \right)^k  
	\frac{1}{k-1} , \\
	\rho_{0,3}(\alpha) &= \int_{\R} m_{\pi^2/\alpha^2,3} (\xi) \ \mathrm{d} \xi . 
\end{align*}
proves Theorem \ref{th:form_asymptotics_2D}.

\subsection*{The asymptotic behaviour  of the ground state eigenvalue of $A_{\ell,b}$}\label{subsec:asympt_formula} 
Recall that  $b \in L_\infty(\R)$ and $\Sigma_\ell = \ell \cdot \Sigma \subseteq \R$, where  $\Sigma \subseteq \R$ is a finite union of bounded intervals. The following theorems provide the asmyptotic behaviour of the ground state eigenvalue
as the window length decreases. 
\begin{theorem}\label{th:main_2D}
	There exists $\ell_0 = \ell_0(\alpha, b, \Sigma) > 0 $ such that  for all $\ell \in (0, \ell_0)$ the operator $A_{\ell,b}$ has a unique
	eigenvalue $\lambda(\ell)$ below its essential spectrum. It satisfies
	\begin{equation}\label{eq:asympt_main_2D}
	\sqrt{\pi^2/\alpha^2  - \lambda(\ell) } 
	=  \ell^2 \left( \frac{\pi^2}{\alpha^3} \right) \cdot \tau_0(\Sigma)  + \mathcal O(\ell^3)  \qquad \text{as} \quad \ell \to 0 .
	\end{equation} 
	The constant $ \tau_0(\Sigma) >0$ is given by \eqref{def:tau_0_2D}. If $b$ is continuously differentiable in some neighbourhood of $0$, then the next term of the asymptotic formula is given by 
	\begin{equation}
		\ell^3 \left( \frac{b(0) \cdot  \tau_1(\Sigma) \cdot \pi^2 }{\alpha^3} \right) 
	\end{equation}
	up to an error of order $\mathcal O(\ell^4 \cdot \ln \ell)$. The constant $\tau_1(\Sigma) >0 $ is given by \eqref{def:tau_1_2D}. 
\end{theorem}
For the special case $\Sigma = (-1, 1)$ we obtain:
\begin{theorem}\label{th:main_2D_special}
	Let $\Sigma_\ell = (-\ell, \ell)$ and let $b$ be twice differentiable in some neighbourhood of $0$. Then the eigenvalue $\lambda(\ell)$ satisfies 
	\begin{align}\label{eq:asympt_main_int}
 	\sqrt{\pi^2/\alpha^2  - \lambda(\ell) } \notag
 	&=  \ell^2 \left( \frac{\pi^3}{2 \alpha^3} \right) + \ell^3 \left( \frac{4b(0) \pi^2 }{3 \alpha^3} \right) 
 		-  \ell^4 \ln \ell \left( \frac{\pi^5}{8 \alpha^5} \right) \\
 	&\quad +  \ell^4 \left(  \frac{\rho_0(\alpha)  \pi^3}{8 \alpha^3 }   +  \frac{\pi^5}{32 \alpha^5}(  1 +  \ln 16)  -   b(0)^2 \cdot 
 		\frac{\rho_1 \cdot  \pi^2}{\alpha^3} \right) + \mathcal O(\ell^5 \ln \ell) 
 \end{align}
as $\ell \to 0$. The constant $\rho_1 >0 $ is given by \eqref{eq:rho_1}. 
\end{theorem}
First we prove the existence and the uniqueness of the eigenvalue of the operator $A_{\ell,b}$ for small  $\ell > 0$.
To this end we  use   the asymptotic expansion in Theorem \ref{th:form_asymptotics_2D}  in its weaker form 
\begin{align}\label{eq:first_asympt_q}
	  \mathcal Q_b(\ell,  \omega) = \frac1\ell  \mathcal Q_0 -  \ell \cdot  \left( \frac{|\Sigma| \cdot \pi^2}{\alpha^3} \right) \cdot  \frac{1}{\sqrt{\pi^2/\alpha^2-\omega}} P_\ct 
		 + R_b(\ell , \omega) 
\end{align}
with the following estimate on the remainder
\begin{align}\label{est:R(ell,omega)}
	\sup\{ \| R_b(\ell , \omega)\|_{\mathcal L (L_2(\Sigma))} : \ell \in (0, \ell_0) \; \wedge \; |\omega - \pi^2/\alpha^2| < \varepsilon \} < \infty .  
\end{align}
We note that $ | \scal{ b T_\ell g}{T_\ell h}_{\Sigma_\ell}  | \le \ell \|b \|_{L_\infty(\Sigma)} \|g\|_{L_2(\Sigma)} \|h \|_{L_2(\Sigma)} $. 
\begin{remark}
	Using a similar argumentation as in Theorem \ref{th:form_asymptotics_2D} it follows 
	that for every compact set $K \subseteq \C \backslash [\pi^2/\alpha^2, \infty)$ there exists $\ell_0  = \ell_0(\alpha, b, \Sigma,  K)$ such that 
	$$ \mathcal Q_b(\ell, \omega) = \frac1\ell \mathcal Q_0 + \tilde R_b(\ell, \omega) , $$
	and the remainder satisfies  $$ \sup \{ \| \tilde R_b(\ell, \omega) \|_{\mathcal L(L_2(\Sigma))} : \omega \in K \; \wedge \ell \in (0, \ell_0 ) \} < \infty . $$
	Recalling that the operator $\mathcal Q_0$ is a invertible, we obtain  
	$$ \ell \mathcal Q_b(\ell, \omega) = \mathcal Q_0 \left( I + \ell \mathcal Q_0^{-1} \tilde R_b(\ell, \omega) \right) . $$
	Choosing $\ell> 0$ sufficiently small implies  that  $\mathcal Q_b(\ell, \omega)$ is invertible for all $\omega \in K$ and $\ell \in (0, \ell_0)$. In particular, $0$ cannot be an eigenvalue of $\mathcal Q_b(\ell, \omega)$. As a consequence 
	the discrete eigenvalues of the operator $A_{ \ell,b}$ converge to $\pi^2/\alpha^2$ as $\ell \to 0$. 
\end{remark}
In what follows we consider for real $\omega$ not only the kernel of the operator $\mathcal Q_{b}(\ell, \omega)$, 
but more generally the discrete eigenvalues of the self-adjoint realisation of  $\mathcal Q_b(\ell, \omega)$.
For $\ell> 0$ and $\omega \in \R  \backslash [\pi^2/\alpha^2, \infty)$ we denote  $$\mu_1(\ell, \omega) \le \mu_2(\ell, \omega) \le \ldots $$ these  eigenvalues counted with multiplicities.
\begin{lemma}\label{lemma:ev_D_ell_2D}
	Let $\ell_0 > 0$ and $\varepsilon > 0$ be  chosen as in Theorem \ref{th:form_asymptotics_2D}.  Then the  following assertions hold true:
	\begin{enumerate}
		\item For fixed   $\ell > 0 $ the function $ \mu_1 (\ell, \cdot )$ is strictly decreasing in $(- \infty, \pi^2/\alpha^2)$. 
		\item For fixed $\ell \in (0,\ell_0)$ we have $ \mu_1 (\ell ,\omega) \to -\infty$ as $\omega \to \pi^2/\alpha^2$. 
		\item For fixed $\omega \in (\pi^2/\alpha^2 - \varepsilon ,  \pi^2/\alpha^2 )$ we have $\mu_1(\ell, \omega) \to \infty$ as $\ell \to 0$. 
		\item  There exists $\tilde \ell_0 = \tilde \ell_0 (\alpha, \Sigma)$ such that for all $\tilde \ell \in (0, \ell_0)$ and for 
		all $| \omega - \pi^2/\alpha^2| < \varepsilon  $ we have $\mu_2(\ell, \omega) > 0$.
	\end{enumerate} 
\end{lemma}
\begin{proof}
	We note that for fixed   $\xi \in \R$ the function $ m_\omega (\xi)$ is strictly decreasing in $\omega$ as can easily be seen from
	\begin{align*}
		m_\omega(\xi) &= \frac1{\alpha} + \sum_{k=1}^\infty \frac{2 \alpha (\xi^2 - \omega)}{\alpha^2 (\xi^2 - \omega) +k^2 \pi^2 } 
		= \frac1{\alpha}  + \sum_{k=1}^\infty \frac{2 \alpha}{\alpha^2  + \frac{k^2 \pi^2}{\xi^2 - \omega} } . 
	\end{align*}
	Thus, for   $- \infty < \omega_1 < \omega_2< \pi^2 /\alpha^2$  and 
	$g \in \tilde H^{1/2}_0(\Sigma) \backslash\{0\}$ we have 
	$$q_b(\ell, \omega_1)[g] > q_b(\ell, \omega_2)[g] . $$ Then assertion (1) follows by applying 
	the min-max principle for self-adjoint operators. 
%	This proves (1). 
	Let us now prove assertion (2). 
	Decomposition \eqref{eq:first_asympt_q}  and  the min-max principle for self-adjoint operators yield for $\ell \in (0, \ell_0)$ and $|\omega - \pi^2/\alpha^2| < \varepsilon$ that $\mu_1(\ell, \omega) \le q_b(\ell, \omega)   [g_0]$ for any  $g_0 \in \tilde  H^{1/2}_0(\Sigma)$ with $\| g_0 \|_{L_2(\Sigma)} =1$. Choosing   $g_0$ such that $\scal{P_\ct g_0}{g_0}_\Sigma  \neq 0$ we obtain 
   	\begin{align*}
   		\mu_1(\ell, \omega)  \le \frac1\ell \scal{\mathcal Q_0 g_0}{g_0}_\Sigma - \ell \cdot  \left( \frac{|\Sigma| \cdot \pi^2}{\alpha^3} \right) \cdot  \frac{1}{\sqrt{\pi^2/\alpha^2-\omega}} \cdot \scal{P_\ct  g_0}{g_0}_{\Sigma}  +  C_1 , 
	\end{align*}
	which tends to $- \infty $ as $\omega \to \pi^2 /\alpha^2$. Here $C_1 := \sup\{ \| R_b(\ell , \omega)\|_{\mathcal L (L_2(\Sigma))} : \ell \in (0, \ell_0) \; \wedge \; |\omega - \pi^2/\alpha^2| < \varepsilon \}$. 
	This proves (2). To deduce (3) we recall that $\mathcal Q_0$ is invertible and we have $q_0[g] = \scal{\mathcal Q_0 g}{g}_\Sigma \ge 0$ for all $g \in \tilde H^{1/2}_0 (\Sigma)$. Thus, there exists $\mu_* > 0$ such that  $$ \scal{\mathcal Q_0 g}{g}_\Sigma = q_0[g] \ge  \mu_* \|g\|_{L_2(\Sigma)}^ 2 , \qquad g \in H^{1/2}_0 . $$ We note
	that the spectrum of the self-adjoint realisation of $\mathcal Q_0$ is discrete since 
	the form domain of $q_0$ is compactly embedded in $L_2(\Sigma)$. Thus,  
	for fixed  $\omega \in \R \backslash [\pi^2/\alpha^2, \infty)$ with  $| \omega - \pi^2/\alpha^2 | < \varepsilon$ we have 
	\begin{align*}
		\mu_1(\ell, \omega) &= \inf\{   q_b(\ell, \omega)[g] : g \in \tilde H^{1/2}_0(\Sigma) \; \wedge \; \| g\|_{L_2(\Sigma)} = 1 \} \ge  \frac{\mu_*}\ell  -  C_1 \to \infty  
	\end{align*}
	as $\ell \to 0$. This proves (3). Assertion (4) follows if we prove that the form $q_b(\ell,\omega)$ is positive on  a subset 
	of codimension $1$. Choose $ g \in \tilde H^{1/2}_0 (\Sigma) $, $\|g\|_{L_2(\Sigma)} = 1$, orthogonal to the constant functions. Then 
	\begin{equation*}
  		 q_b(\ell , \omega) [g]  = \frac1\ell q_0[g]  +  \scal{R_b(\ell, \omega)g}{g}_\Sigma   \ge \frac{\mu_*}{\ell} - C_1  > 0
	\end{equation*}
	for $0 < \ell < \tilde \ell_0 := \min \{ 1, \mu_*/C_1 \}$ and $|\omega - \pi^2/\alpha^2| < \varepsilon$. This concludes the proof of Lemma \ref{lemma:ev_D_ell_2D}.
\end{proof}
\begin{lemma}\label{lemma:uniqueness_2D}
	There exists $\ell_0 = \ell_0(\alpha, b,  \Sigma) > 0$ such that for all $\ell \in (0, \ell_0)$ the 
	operator $A_{\ell,b}$ has a unique eigenvalue $\lambda(\ell)$ below its
	essential spectrum.
\end{lemma}
\begin{proof}
	We start by proving the uniqueness of the eigenvalue. Let $\varepsilon> 0$ be chosen as in Theorem \ref{th:form_asymptotics_2D} and  Lemma \ref{lemma:ev_D_ell_2D}. 
	Using the remark before Lemma \ref{lemma:ev_D_ell_2D} we may choose $\ell_0 > 0$ such that 
	$\inf \sigma(A_{\ell,b}) \ge \pi^2/\alpha^2 - \varepsilon$
	for all $\ell \in (0, \ell_0)$. %
	Moreover, we assume  that $\mu_2(\ell, \omega) > 0$ for
	all $ \ell \in (0,\ell_0)$ and 	$\omega \in (\pi^2/\alpha^2 - \varepsilon , \pi^2/\alpha^2)$. Fix $\ell \in (0, \ell_0)$ and assume that
	$\omega \in \sigma_d(A_{\ell,b})$. Then $\mu_1(\ell,\omega) = 0$.
	Lemma \ref{lemma:ev_D_ell_2D} (1) implies for $\omega_1 < \omega< \omega_2< \pi^2/\alpha^2$   
	$$ \mu_1 (\ell,\omega_1) < \mu_1 (\ell,\omega) = 0 < \mu_1 (\ell,\omega_2)  . $$
	In particular we have 
	$\mathrm{ker} \; \mathcal  Q_b(\ell, \omega_1) = \mathrm{ker} \; \mathcal Q_b(\ell, \omega_2) = \{0\}$,  which  proves the 
	uniqueness of the eigenvalue of $A_{\ell,b}$. 
	
	Next we prove the existence of the eigenvalue. Using Lemma \ref{lemma:ev_D_ell_2D} (3) we may assume that 
	$\mu_1(\ell, \pi^2/\alpha^2  - \varepsilon/2) > 0$ for all $\ell \in (0,  \ell_0)$. Fix  $ \ell \in (0,  \ell_0)$. Since $\mu_1(\ell, \omega) \to  - \infty $ as $\omega \to 
	\pi^2 /\alpha^2$ and $\mu_1(\ell, \omega)$ depends continuously on $\omega$
	it follows that there exists $\tilde \omega  \in (\pi^2/\alpha^2 - \varepsilon/2 , \pi^2/\alpha)$ such that $\mu_1(\ell, \tilde \omega) = 0$. 
\end{proof}
\begin{remark}
	Another  method of proof for  Lemma \ref{lemma:uniqueness_2D} may be  based on a variant  of  operator-valued Rouch\'e's theorem, cf.\ e.g.\ \cite{GohbergSigal} or the monograph \cite{AmKaLee}.
\end{remark}
Next we prove the asymptotic formula for the eigenvalue of $A_{\ell,b}$ using the Birman-Schwinger principle. To this end  we choose  $\ell_0 > 0$ such  that the operator 
$ \mathcal Q_0 + \ell R_b(\ell, \omega) $
is invertible for all $\ell \in (0, \ell_0)$ and  $\omega \in (\pi^2/\alpha^2 - \varepsilon, \pi^2/\alpha^2)$. The existence of such an $\ell_0$ follows from the estimate \eqref{est:R(ell,omega)}.
\begin{lemma}[Birman-Schwinger principle]\label{lemma:Birman_Schwinger}
	Let $\ell \in (0, \ell_0)$ and  $\omega \in (\pi^2/\alpha^2 - \varepsilon, \pi^2/\alpha^2)$. Then $0$ is an eigenvalue of 
	the operator 
	$$ \ell \mathcal Q_b(\ell, \omega) = \mathcal Q_0   - \ell^2 \left( \frac{|\Sigma| \cdot \pi^2}{\alpha^3} \right) 
		\frac{1}{\sqrt{\pi^2/\alpha^2-\omega}} P_\ct +\ell R_b(\ell, \omega)    $$
	if and only if $1$ is an eigenvalue of the Birman-Schwinger operator 
	\begin{equation} \label{eq:BS_2D}
		\ell^2 \left( \frac{|\Sigma|  \cdot \pi^2}{\alpha^3} \right) 
		\frac{1}{\sqrt{\pi^2/\alpha^2-\omega}} \cdot  P_\ct^{1/2}  (  \mathcal Q_0  +\ell R_b(\ell, \omega) )^{-1} P_\ct^{1/2} 
	\end{equation}
\end{lemma}
A proof may be found e.g. in \cite{BirSol92}. 

Since  the projection $P_\ct$ is a rank-one operator with  $P_\ct^2 = P_\ct = P_\ct^{1/2}$, 
the Birman-Schwinger principle implies that $\omega$ is an eigenvalue of the operator $A_{\ell,b}$ if and only if 
the trace of the Birman-Schwinger operator is equal 
to one, i.e.,
$$ \ell^2 \left( \frac{\pi^2}{\alpha^3} \right) \frac{1}{\sqrt{\pi^2/\alpha^2-\omega}}
	\scal{(  \mathcal Q_0  +\ell R_b(\ell, \omega)   )^{-1} \phi_0}{\phi_0 }_\Sigma = 1 ,$$
where $\phi_0(x) = 1$ is the (non-normalised)  constant function on $L_2(\Sigma)$. For the
choice $\omega = \lambda(\ell)$ we obtain 
\begin{equation}\label{eq:eigenvalue_Birman_Schwinger_2D}
	\sqrt{\pi^2/\alpha^2  - \lambda(\ell) } 
	=  \ell^2 \left( \frac{ \pi^2}{\alpha^3} \right) \scal{( \mathcal Q_0  +\ell R_b(\ell, \lambda(\ell)))^{-1} \phi_0 }{\phi_0  }_\Sigma .
\end{equation}
Next we use an asymptotic expansion for the resolvent term. We have 
	\begin{align}
		( \mathcal Q_0  +\ell R_b(\ell, \omega) )^{-1} &= ( I  + \ell  \mathcal  Q_0^{-1}  R_b(\ell, \omega))^{-1}  \mathcal  Q_0^{-1}   , \notag  \\
			&= \sum_{k=0}^\infty \ell^k \bigr( - \mathcal  Q_0^{-1}  R_b(\ell, \omega) \bigr)^{k} \mathcal Q_0^{-1}   = \mathcal Q_0^{-1}  + \mathcal O(\ell) , \label{eq:exp_Q+R_ell}
	\end{align}
uniformly in $\omega \in (\pi^2/\alpha^2, - \varepsilon, \pi^2/\alpha^2)$. Note that  for sufficiently small  $\ell$ the sum converges absolutely  in $\mathcal L(L_2(\Sigma))$. % and $\mathcal 
Hence, 
$$ \scal{(\mathcal Q_0  +\ell R_b(\ell, \lambda(\ell) ) )^{-1} \phi_0 }{\phi_0 }_{\Sigma} = \scal{\mathcal Q_0^{-1} 
	\phi_0 }{ \phi_0 }_{\Sigma} +  \mathcal{O}(\ell)  $$
as $\ell \to 0$, and thus, 
\begin{align*}
	\sqrt{\pi^2/\alpha^2  - \lambda(\ell) } 
	&= \ell^2 \left( \frac{\pi^2}{\alpha^3} \right) \scal{( \mathcal Q_0 + \ell R_b(\ell,  \lambda(\ell)) )^{-1} \phi_0 }{ \phi_0  }_{\Sigma} \\
	&= \left( \frac{\pi^2}{\alpha^3} \right) \cdot  \scal{\mathcal Q_0^{-1} \phi_0}{\phi_0}_\Sigma \cdot  \ell^2    + \mathcal{O}(\ell^3) .
\end{align*}
This proves the first term of the asymptotics in Theorem \ref{th:main_2D} with 
\begin{align}\label{def:tau_0_2D}
	\tau_0(\Sigma) := \scal{\mathcal Q_0^{-1} \phi_0}{\phi_0}_\Sigma = \scal{\mathcal Q_0^{-1/2} \phi_0}{Q_0^{-1/2}  \phi_0}_\Sigma > 0 .
\end{align}
Until now  no additional assumptions
on $b \in L_\infty(\R)$ were necessary. Now let $b$ be differentiable in a neighbourhood of $0$. 
Then 
$$ \scal{b T_\ell g}{T_\ell h}_{\Sigma_\ell  } = \int_{\Sigma} b(\ell x) \cdot g(x) \; \overline{h(x)}  \ \mathrm{d}  x = b(0) \scal{ g}{h}_{\Sigma} + \mathcal O(\ell) . $$
Note that the remainder may be estimated uniformly in the operator norm. Thus, together with Theorem \ref{th:form_asymptotics_2D} we obtain %for the operator $R_b(\ell, \omega)$ 
\begin{align*}
	R_b(\ell, \omega) =  b(0) I   + \mathcal O(\ell \ln \ell )  .
\end{align*}
The estimate holds uniformly in $\ell \in (0, \ell_0)$ and $| \omega - \pi^2/\alpha^2| < \varepsilon$. 
Using Formula \eqref{eq:exp_Q+R_ell} we obtain 
\begin{align*}
	( \mathcal Q_0 + \ell R_b(\ell, \lambda(\ell) )^{-1} &= \mathcal Q_0^{-1} + \ell \mathcal Q_0^{-1} R_b(\ell, \lambda(\ell)) \mathcal Q_0^{-1} + \mathcal O(\ell^2) \\[4pt]
	&= \mathcal Q_0^{-1} + \ell \cdot b(0) \cdot I + \mathcal O(\ell^2 \cdot \ln \ell)  . 
\end{align*}
Finally, 
\begin{align*}
  \sqrt{\pi^2/\alpha^2  - \lambda(\ell) } 
	 &= \ell^2 \left( \frac{\pi^2}{\alpha^3} \right) \scal{( \mathcal Q_0 + \ell R_b(\ell, \omega) )^{-1} \phi_0 }{ \phi_0  }_{\Sigma} \\
	&= \left( \frac{\pi^2}{\alpha^3} \right) \cdot  \tau_0(\Sigma) \cdot  \ell^2   + \left( \frac{b(0) \pi^2}{\alpha^3} \right) \cdot \tau_1(\Sigma) \cdot  \ell^3  + \mathcal O(\ell^4 \ln \ell) 
\end{align*}
with 
\begin{equation}\label{def:tau_1_2D} 
	\tau_1(\Sigma) := \scal{\mathcal Q_0^{-2} \phi_0}{\phi_0}_\Sigma = \scal{\mathcal Q_0^{-1} \phi_0}{\mathcal Q_0^{-1} \phi_0}_\Sigma > 0 . 
\end{equation}
This proves Theorem \ref{th:main_2D}. 

Now let $\Sigma =(-1,1)$. Then  the operator $\mathcal Q_0$ becomes the composition of  the standard finite Hilbert transform and the  derivative.
Using \cite[Formula (4.8)]{AmKaLeeElastic13} or \cite[Section 5.2]{AmKaLee} we obtain
$$ (\mathcal Q_0^{-1} \phi_0 )(x) =\sqrt{1- x^2} , $$
which implies
\begin{align*}
	\scal{\mathcal Q_0^{-1}  \phi_0 }{\phi_0 }_{\Sigma} =  \int_{-1}^1 \sqrt{1-x^2} \ \mathrm{d} x = \frac\pi2 .
\end{align*}
For the sake of simplicity  we assume that $b \in C^2(\R)$. Then 
$$ \scal{b T_\ell g}{T_\ell h}_{\Sigma_\ell  } = \int_{\Sigma} b(\ell x) \cdot g(x) \; \overline{h(x)}  \ \mathrm{d}  x = b(0) \scal{ g}{h}_{\Sigma} + 
	  \ell \cdot b ' (0) \scal{M_x g}{h}_{\Sigma}   + \mathcal O(\ell^2) , $$
where $M_x : L_2(\Sigma) \to L_2(\Sigma)$ is the   multiplication operator $(M_xf)(x) = x f(x) $. Theorem \ref{th:form_asymptotics_2D} implies 
\begin{align}\label{eq:exp_remainder_2D}
	\begin{split} R_b(\ell, \omega) &=  b(0)  \cdot I +  \ell \ln \ell \cdot \frac{ \pi}{\alpha^2} P_\ct 
	+ \ell \left(  \frac{\rho_0(\alpha)}\pi  P_\ct +   \frac{\pi }{2\alpha^2} K_{\ln|x| } +  b'(0) M_x  \right) \\
	& \quad+ R^{(1)}_b(\ell, \omega), 
	\end{split}
\end{align}
with 
$$ \| R_b^{(1)}(\ell, \lambda(\ell) ) \|_{\mathcal L(L_2(\Sigma))} \le C ( \ell^3 \ln \ell  + \pi^2/\alpha^2 - \lambda(\ell) ) = \mathcal O(\ell^3 \ln \ell) . $$
To calculate the asymptotic behaviour of the eigenvalue we use the expansion
\begin{align*}
	( \mathcal Q_0 + \ell R_b(\ell, \lambda(\ell) )  )^{-1}  &= \mathcal Q_0^{-1} - \ell \cdot b(0) \mathcal Q_0^{-2}  -  \ell^2 \ln \ell \cdot \frac{\pi   }{\alpha^2} \mathcal Q_0^{-1}  P_\ct \mathcal Q_0^{-1} \\
	&\quad - \ell^2 \cdot  \mathcal Q_0^{-1}  \left( \frac{\rho_0(\alpha)}\pi   P_\ct 
		+   \frac{\pi }{2\alpha^2} K_{\ln|x|}  +  b'(0) M_x   \right) \mathcal Q_0^{-1} \\[6pt]
	&\quad + \ell^2  \cdot  b(0)^2  \mathcal Q_0^{-3}  + \mathcal O(\ell^3 \ln \ell ) . 
\end{align*}
Then 
\begin{align*}
	&\ell^{-2} \left( \frac{\pi^2}{\alpha^3} \right)^{-1} \cdot \sqrt{\pi^2/\alpha^2  - \lambda(\ell) } \\
	& \quad= \frac\pi2 - \ell \cdot b(0) \scal{\mathcal Q_0^{-2}\phi_0}{\phi_0}_{(-1,1)}  -  \ell^2 \ln \ell \cdot 
		\frac{\pi}{\alpha^2} \scal{ \mathcal Q_0^{-1} P_\ct \mathcal Q_0^{-1} \phi_0}{\phi_0}_{(-1,1)}  \\[6pt]
	&\qquad - \ell^2 \cdot \frac{\rho_0(\alpha)}{\pi} \scal{\mathcal Q_0^{-1} P_\ct \mathcal Q_0^{-1} \phi_0}{\phi_0}_{(-1,1)} 
		-  \ell^2 \cdot \frac{\pi}{\alpha^2} \scal{\mathcal Q_0^{-1} K_{\ln|x|} \mathcal Q_0^{-1} \phi_0}{\phi_0}_{(-1,1)} \\[6pt]
	&\qquad - \ell^2 \cdot b'(0) \scal{\mathcal Q_0^{-1} M_x \mathcal Q_0^{-1} \phi_0}{\phi_0}_{(-1,1)} 
		+  \ell^2 \cdot b(0)^2 \scal{\mathcal Q_0^{-3} \phi_0}{\phi_0}_{(-1,1)} \\[6pt]
	&\qquad + \mathcal O(\ell^3 \ln \ell  ) . 
\end{align*}
In order to calculate the asymptotic behaviour of $\scal{(\mathcal Q_0^{-1} + \ell R_{\ell, \lambda(\ell)} )^{-1} \phi_0 }{\phi_0}_{(-1,1)}$ we shall need
the following identities 
\begin{align*}
	\scal{\mathcal Q_0^{-2 } \phi_0}{\phi_0}_{(-1,1)} &= \scal{\mathcal Q_0^{-1} \phi_0}{\mathcal Q_0^{-1}\phi_0}_{(-1,1)}
		= \int_{-1}^1 ( 1-x^2 ) \ \mathrm{d} x = \frac43 ,  \\
	\scal{\mathcal Q_0^{-1}  P_\ct \mathcal Q_0^{-1} \phi_0 }{\phi_0}_{(-1,1)} &= \scal{P_\ct \mathcal Q_0^{-1} \phi_0}{\mathcal Q_0^{-1} \phi_0}_{(-1,1)} 
	= \frac{\pi^2}8  , \\
	\scal{\mathcal Q_0^{-1} M_x  \mathcal Q_0^{-1} \phi_0 }{\phi_0}_{(-1,1)} &= \scal{ M_x  \mathcal Q_0^{-1} \phi_0 }{\mathcal Q_0^{-1} \phi_0}_{(-1,1)}
		= 	\int_{-1}^1 x ( 1- x^2)  \ \mathrm{d} x = 0. 
\end{align*}
Next we calculate $\scal{K_{\ln|x|} \mathcal Q_0^{-1} \phi_0}{\mathcal Q_0^{-1}  \phi_0}_{(-1,1)}$. Recall that $(\mathcal Q_0^{-1} \phi_0)(x) = \sqrt{1-x^2}$. 
Using \cite[Formulae (5.6)-(5.9)]{AmKaLee} for  $\psi := K_{\ln|x|} \mathcal Q_0^{-1} \phi_0$   we obtain 
$$ \psi'(x) = \mathrm{p.v.} \int_{-1}^1 \frac{\phi_0(x)}{x-y} \ \mathrm{d} y = \mathrm{p.v.} \int_{-1}^1 \frac{\sqrt{1-x^2}}{x-y} \ \mathrm{d} y = \pi x . $$
Thus, $\psi(x) = \frac{\pi x^2}{2} + \psi(0) $, where 
$$ \psi(0) = \int_{-1}^1 \ln |x| \sqrt{1-x^2} \ \mathrm{d} x  = 2 \int_{0}^1 \ln |x| \sqrt{1-x^2} \ \mathrm{d} x = - \frac{\pi}2 \left(\frac12 + \ln 2 \right) , $$
cf. \cite[Section 4.241]{GradRyz}. Hence, 
\begin{align*}
	&\scal{K_{\ln|x|} \mathcal Q_0^{-1} \phi_0}{\mathcal Q_0^{-1} \phi_0}_{(-1,1)} \\
	&\qquad = \frac\pi2 \int_{-1}^1 x^2 \sqrt{1-x^2} \ \mathrm{d} x  - \frac{\pi}2 \left(\frac12 + \ln 2 \right) \int_{-1}^1 \sqrt{1-x^2} \ \mathrm{d} x 
	= \frac{\pi^2}{16} \left(  - 1  - \ln 16\right) . 
\end{align*}
Setting  
\begin{align}\notag
	\rho_1 &:= \scal{\mathcal Q_0^{-3} \phi_0}{\phi_0}_{(-1,1)} = \scal{\mathcal Q_0^{-2} \phi_0}{\mathcal Q_0^{-1} \phi_0}_{(-1,1)} \\
	&= \scal{\mathcal Q_0^{-3/2} \phi_0}{\mathcal Q_0^{-3/2} \phi_0}_{(-1,1)} > 0  \label{eq:rho_1}
\end{align}
we obtain
\begin{align*}
	\sqrt{\pi^2/\alpha^2  - \lambda(\ell) } 
	&=  \ell^2 \left( \frac{\pi^3}{2 \alpha^3} \right) - \ell^3 \left( \frac{4b(0) \pi^2}{3 \alpha^3} \right) 
		- \ell^4 \ln \ell \left( \frac{\pi^5}{8 \alpha^5} \right) \\
	&\quad  - \ell^4 \left(  \frac{\rho_0(\alpha)  \pi^3}{8 \alpha^3 } -   \frac{\pi^5}{32 \alpha^5}(  1 +  \ln 16)  -   b(0)^2 \rho_1  \cdot 
		\frac{ \pi^2}{\alpha^3} 
		\right) + \mathcal O(\ell^5 \ln \ell).
\end{align*}
This proves Theorem \ref{th:main_2D_special}. 

Concluding the two-dimensional case 
 we briefly want to  sketch what happens in the case of two  waveguides of width $\alpha_+$ and  $\alpha_-$, which are coupled through a window
$\Sigma_\ell := \ell \cdot \Sigma$. We use the same ansatz and  introduce  the corresponding Dirichlet-to-Neumann operators $D_{\ell,\omega}^{+}$ and $D_{\ell,\omega}^{-}$ 
on the upper and on the  lower waveguide. Comparing the normal derivatives along the window we observe that $\omega$ is an eigenvalue of the corresponding Dirichlet-Laplacian 
if and only if $0$ is an eigenvalue of 
$$ D_{\ell, \omega} := D_{\ell, \omega}^+  + D_{\ell, \omega}^{-} . $$ 
Using the same scaling operator $T_\ell$ as above leads  to the analysis of the operator 
$$ \mathcal Q(\ell, \omega) = \mathcal Q^+(\ell, \omega) + \mathcal Q^-(\ell, \omega) =  T_\ell^* D_{\ell, \omega}^+ T_\ell 
	+  T_\ell^* D_{\ell, \omega}^- T_\ell  $$
In what follows we assume that $\alpha_+ > \alpha_-$ so that  the essential spectrum  of the corresponding operator $A_\ell$ is given by the interval
$[\pi^2 /\alpha_+^2, \infty)$. Using the asymptotic expansions of $\mathcal Q^\pm(\ell , \omega)$ as $\ell \to 0$ and $\omega \to \pi^2/\alpha^2_+$ we obtain 
$$ \mathcal Q(\ell, \omega )
	= \frac2\ell  \mathcal Q_0 - \frac{\ell}{\sqrt{\pi^2 /\alpha^2 - \omega}} \left( \frac{|\Sigma| \cdot \pi^2}{\alpha_+^3} \right)  P_\ct 
	+ \mathcal O(1)  ,
$$
The same approach yields now the following result.
\begin{theorem}\label{th:main_coupled}
	In the case of two coupled waveguides the ground state eigenvalue $\lambda(\ell)$ satisfies
	\begin{align}
		\sqrt{ \pi^2/\alpha^2 - \lambda(\ell)} = \left( \frac{\pi^2}{2 \alpha_+^3} \right) \cdot \tau_0(\Sigma) \cdot 
		\ell^2    + \mathcal{O}(\ell^3 ) \qquad \text{as} \quad \ell \to 0 . 
	\end{align}
	Here $\tau_0 (\Sigma) > 0 $ is again given by \eqref{def:tau_0_2D}. 
\end{theorem}
\section{Infinite layers}\label{sec:3D}
We consider the mixed problem for  an infinite layer $\Omega := \R^2 \times (0, \alpha)$
with coordinates $(x,y) = (x_1, x_2, y) \in \R^2 \times (0, \alpha)$. Let 
$\Sigma \times  \{0 \} \subseteq \partial \Omega$ be the Robin window, where $\Sigma \subseteq \R^2$ is a  bounded open subset 
with Lipschitz boundary.  For $\ell > 0$ we denote by    $\Sigma_\ell  := \ell \cdot \Sigma \subseteq \R^2$  the scaled window. 
Let  $b \in L_\infty(\R^2)$ be a real-valued function and consider the  quadratic form
\begin{equation}
	a_{\ell,b} [u ] := \int_{\Omega} |u(x,y)|^2 \ \mathrm{d} x \ \mathrm{d} y + \int_{\Sigma_{\ell}} b(x)  \cdot |u(x,0)|^2  \ \mathrm{d} x  
\end{equation}
with the form domain 
\begin{equation}
	D[a_{\ell,b}] := \left\{ u \in H^1(\Omega) : u|_{\R^2 \times \{\alpha \} } = 0 \; \wedge \;
	\mathrm{supp} (u|_{\R^2 \times \{0\}} ) \subseteq \overline{\Sigma_\ell} \right\} . 
\end{equation}
As in the two-dimensional case we observe that $a_{\ell,b}$ is 
a closed semi-bounded form in $L_2(\Omega)$, and thus, it induces a self-adjoint operator
$A_{\ell,b}$. The essential spectrum of $A_{\ell,b}$ is independent of $b$ and $\ell$ and given by 
$\sigma_{\mathrm{ess}} (A_{\ell,b} )   = [ \pi^2/\alpha^2, \infty) . $
We prove the following theorem. 
\begin{theorem}\label{th:main_3D}
	There exists $\ell_0 = \ell_0(\alpha, b, \Sigma) >0 $ such that the operator $A_{\ell,b}$ has a unique eigenvalue $\lambda(\ell)$ below the essential 
	spectrum $[\pi^2/\alpha^2, \infty)$. If $b$ is $C^1$ in some neighbourhood of $0 \in \R^2$ then the  eigenvalue satisfies the asymptotic estimate
	\begin{equation}
		\ln (\pi^2/\alpha^2 - \lambda(\ell) ) = - \ell^{-3} \frac{4\alpha^3}{ \tau_0 (\Sigma) + \tau_1 (\Sigma) b(0) \ell +  \mathcal{O}(\ell^2)} 
		\qquad \text{as} \quad  \ell \to 0, 
	\end{equation}
	with  constants $\tau_0(\Sigma) >0 $ given by \eqref{def:tau_0_3D} and $\tau_1(\Sigma) > 0 $   given by \eqref{def:tau_1_3D}.
\end{theorem}
Since we shall only slightly modify our approach we will merely sketch the   major steps of the proof. 
Actually, most of the results proven in the two-dimensional case may be reused.
Let $\omega \in \C$ and $g \in H^{1/2}(\R^2)$. We consider 
for $u \in H^1(\Omega)$ the  Poisson problem 
\begin{equation}\label{eq:Poisson_3D}
	( -\Delta - \omega )u  =0  \quad \text{in } \Omega , \qquad u (\cdot , 0 ) = g , \qquad u(\cdot , \alpha ) = 0 . 
\end{equation}
Applying  the Fourier transform with respect to the first two variables  
leads for every $\xi= (\xi_1, \xi_2) \in \R^2$  to the following Sturm-Liouville problem
$$ (- \partial_y^2 + |\xi|^2 - \omega) \hat u(\xi,\cdot ) = 0 , \quad \text{in } (0,\alpha) \qquad \hat u (\xi  , 0) = \hat g(\xi), \qquad 
\hat u (\xi, \alpha) = 0 , $$
where  $\xi \in \R^2$. The  solution of \eqref{eq:Poisson_3D} is given by 
\begin{equation}
	\hat u (\xi, y) := \hat g(\xi) \cdot  \frac{\sinh((\alpha - y) \sqrt{|\xi|^2  - \omega})}{\sinh(\alpha \sqrt{|\xi|^2 - \omega})} , 
\end{equation}
which is similar to Formula  \eqref{eq:sol_u}. 
In the same way  we obtain that  the Poisson problem  \eqref{eq:Poisson_3D} is uniquely solvable for all $g \in H^{1/2}(\R^2)$ if and only if 
$\omega \in \C \backslash [\pi^2/\alpha^2 , \infty)$. Moreover, 
there exists a constant
$c = c(\alpha, \omega) > 0 $ such that $ \| u\|_{H^1(\Omega)} \le c \| g\|_{H^{1/2}(\R^2)} .  $
In what follows let  $\omega \in \C \backslash [\pi^2/\alpha^2 , \infty)$.
Then the normal derivative of $u$ satisfies
$$ \widehat{\partial_n u} (\xi, 0 ) = m_{\omega} (|\xi|) \cdot \hat g (\xi) , \quad \xi \in \R^2, $$
where the function $m_\omega$ is defined  as in the two-dimensional case, i.e.,
\begin{equation}
	 m_\omega(|\xi|) = \sqrt{|\xi|^2 - \omega} \cdot \coth(\alpha\sqrt{|\xi|^2 - \omega}) . 
\end{equation}
Hence, the  Dirichlet-to-Neumann operator for the infinite layer is given by the Fourier integral operator
\begin{equation}
	D_\omega : H^{1/2} (\R^2) \to H^{-1/2}(\R^2) , \qquad \widehat{D_\omega g} (\xi) := m_\omega(|\xi|) \cdot \hat g(\xi) .
\end{equation}

The next step is to define  the truncated operator on the boundary. The corresponding spaces 
$\tilde  H^{1/2}_0(\Sigma_\ell)$ and $H^{-1/2}(\Sigma_\ell)$ are defined  as in \eqref{def:Hs_0} and \eqref{def:Hs}.
As both $\Sigma$ and  $\Sigma_\ell$ have Lipschitz boundary, the dual pairing 
\eqref{eq:dual_pairing} still holds true, cf.\ \cite[Theorems 3.14, 3.30]{McLean}. 
Put 
\begin{equation}
	D_{\ell, \omega} +b : \tilde H_0^{1/2}(\Sigma_\ell)  \to H^{-1/2}(\Sigma_\ell), \qquad 
		D_{\ell, \omega} +b := r_\ell (D_\omega + b ) e_\ell    , 
\end{equation}
where $r_\ell : H^{-1/2}(\R^2) \to H^{-1/2}(\Sigma_\ell)$ denotes the restriction operator
and $e_\ell :\tilde H_0^{1/2}(\Sigma_\ell) \to H^{1/2}(\R^2)$ the embedding operator. 
As in Lemma \ref{lemma:kernel} we obtain 
\begin{equation}
	\mathrm{dim}\; \mathrm{ker} (A_{\ell,b} - \omega) = \mathrm{dim}\; \mathrm{ker} (D_{\ell,\omega} + b) .  
\end{equation}
Let 
$$ T_\ell : L_2(\Sigma) \to L_2(\Sigma_\ell) , \qquad (T_\ell g )(x) := \ell^{-1} g(x/\ell) . $$
be the unitary scaling operator. In what follows we consider  
the scaled operator
\begin{equation} \mathcal Q_b (\ell, \omega) =  T_\ell^* (D_{\ell, \omega} + b ) T_\ell \end{equation}
together with its associated sesquilinear form 
\begin{align}
	q_b (\ell , \omega) &:= \scal{\mathcal Q_b(\ell, \omega) g}{h}_{\Sigma} \\
	&= \ell^2 \int_{\R^2}  m_\omega(|\xi|) \cdot \hat g(\ell \xi) \; \overline{h(\ell \xi)} \ \mathrm{d} \xi + 
	\int_{\Sigma} b(\ell x) \cdot g(x) \; \overline{h(x)} \ \mathrm{d} x ,
\end{align}
where $g, h \in D[q_b(\ell,\omega) ] := \tilde H^{1/2}_0(\Sigma)$. 
As before  we define $\mathcal Q_0 : \tilde H^{1/2}_0(\Sigma) \to H^{-1/2}(\Sigma)$,
\begin{equation} \scal{\mathcal Q_0 g}{h}_\Sigma := q_0 [g,h] := \int_{\R^2}  |\xi| \cdot \hat g(\xi) \cdot \overline{\hat h(\xi)} \ \mathrm{d} \xi , \end{equation}
and let $P_\ct$ denote the projection onto  the space of constant functions in $L_2(\Sigma)$. Moreover,  we
denote by  $K_{\frac1{|x|}} : L_2(\Sigma) \to L_2(\Sigma) $  the following convolution operator 
\begin{equation}
	(K_{\frac{1}{|x|}} f)(z) := \int_{\Sigma} \frac{f(z)}{|x-z|} \ \mathrm{d} x, \qquad z \in \Sigma .  
\end{equation}
\begin{theorem}\label{th:form_asymptotics_3D}
	Let  $b=0$. There exists $\ell_0 >0 $ and $\varepsilon > 0$ such that for $\ell \in (0, \ell_0)$ and $|\omega - \pi^2/\alpha^2| < \varepsilon$ the following expansion holds true 
  	\begin{align} 
		\mathcal Q_0(\ell, \omega) &= \frac1\ell  \mathcal Q_0  +  
  			\ell^2  \cdot \frac{|\Sigma|}{4 \alpha^3} \ln (\pi^2/\alpha^2 - \omega )  P_\ct - \ell \frac{\pi}{4 \alpha^2} K_{\frac1{|x|}} +  R(\ell, \omega) . 
  	\end{align}
  	Here $|\Sigma|$ denotes  the volume of $\Sigma$ and the remainder satisfies  $$\|R(\ell,\omega)\|_{\mathcal L (L_2(\Sigma))}  \le C (\ell^2 + \pi^2 / \alpha^2 - \omega)$$
  	for some constant  $C = C( \alpha, \Sigma) > 0 $ which is independent of  $\ell$,  $\omega$.
\end{theorem}
\begin{proof}
	 We use the same decomposition for  $q_0(\ell,\omega)$ as in the two-dimensional case and put
	\begin{align*}
		q_0(\ell, \omega)[g,h] &= 
		\ell^2 \left(\int_{\{|\xi| \le  1  \} } + \int_{\{|\xi| > 1 \}} \right) m_\omega(|\xi|) \cdot \hat g(\ell \xi) \; \overline{h(\ell \xi)}  \ \mathrm{d} \xi  \\
		&=: q_0^{(1)}(\ell,\omega) [g,h] + q_0^{(2)}(\ell,\omega)[g,h]  .
	\end{align*}
	Recall that
	\begin{equation*}
		m_\omega(\xi ) = \frac1{\alpha} + \sum_{k=1}^\infty \frac{2 \alpha (\xi^2 - \omega)}{\alpha^2 (\xi^2 - \omega) +k^2 \pi^2 }  = - \left(\frac{2\pi^2}{\alpha^3}\right)  \frac{1}{\xi^2- \omega +\pi^2/\alpha^2} + \mathcal O(1)
	\end{equation*}
	and thus,  
	\begin{align*}
		q_0^{(1)}(\ell,\omega) [g,h] &= - \ell^2 \cdot \left(\frac{2\pi^2}{\alpha^3}\right) \int_{\{ |\xi| \le 1\}} \frac{1}{|\xi|^2- \omega +\pi^2/\alpha^2}   \cdot \hat g(\ell \xi) \; \overline{h(\ell \xi)}  \ \mathrm{d} \xi  + \mathcal O(\ell^2) .
	\end{align*}
	We note that the the first expression coincides  almost with the free resolvent of the Laplacian in $\R^2$, which with respect to  the spectral parameter $\omega$ has a logarithmic singularity. 
	Using the Taylor expansion of $\hat g \cdot \overline{\hat h}$ at $0$ we have 
	\begin{align*}
		&- \ell^2 \cdot \left(\frac{2\pi^2}{\alpha^3}\right) \int_{\{ |\xi| \le 1\}} \frac{1}{|\xi|^2- \omega +\pi^2/\alpha^2}   \cdot \hat g(\ell \xi) \; \overline{h(\ell \xi)}  \ \mathrm{d} \xi \\
		&\quad= - \ell^2 \cdot \left(\frac{2\pi^2}{\alpha^3}\right) \cdot \sum_{\beta \in \N_0^2 } \ell^{|\beta|} \frac{1}{\beta!} \left. \frac{\partial^\beta}{\partial \xi^\beta} \Bigr(\hat g(\xi)  \overline{\hat h} (\xi) \Bigr) \right|_{\xi=0} \cdot
		\int_{\{|\xi| \le 1\}} \frac{\xi^{\beta+1}}{|\xi|^2 + \pi^2/\alpha^2-  \omega} \ \mathrm{d} \xi \\
		&\quad= - \ell^2 \cdot \left(\frac{2\pi^2}{\alpha^3}\right) \hat g(0)  \overline{\hat h} (0) \cdot
		\int_{\{|\xi| \le 1\}} \frac{1}{|\xi|^2 + \pi^2/\alpha^2-  \omega} \ \mathrm{d} \xi  + \mathcal O(\ell^3) ,
	\end{align*}
	since for $|\beta| \ge 1$ we have 
	$$  \left| \int_{\{|\xi| \le 1\}} \frac{\xi^\beta }{|\xi|^2 + \pi^2/\alpha^2-  \omega} \ \mathrm{d} \xi  \right| \le \int_0^1 \frac{r^2}{r^2 + \pi^2/\alpha^2 - \omega} \ \mathrm{d} r \le C $$
	and  $C$ may be chosen independently of $\omega$. Moreover, 
	$$  \int_{\{|\xi| \le 1\}} \frac{1}{|\xi|^2 + \pi^2/\alpha^2-  \omega} \ \mathrm{d} \xi = - \frac{1}2 \ln \left(\pi^2/\alpha^2 - \omega\right) + \mathcal O(1) , $$
	and thus, 
	\begin{align*} 
		q_0^{(1)}(\ell,\omega) [g,h]  &=  \ell^2 \cdot \left(\frac{\pi^2}{\alpha^3}\right)  \hat g(0) \cdot \overline{\hat h} (0) \cdot \ln \left(\pi^2/\alpha^2 - \omega\right) + \mathcal O(\ell^3) \\
		&= \ell^2 \cdot \left(\frac{|\Sigma| }{4 \alpha^3}\right) \ln \left(\pi^2/\alpha^2 - \omega\right)  \scal{P_\ct g}{h}_{\Sigma} + \mathcal O(\ell^3)  .
	\end{align*}
	Next we consider   the form $q_0^{(2)}(\ell,\omega )$. The expansion \eqref{eq:exp_m_omega_symbol}  of $m_\omega$ for large $|\xi|$  implies 
	\begin{align*}
		 & q_0^{(2)}(\ell,\omega) [g,h] \\ &= \ell^2 \int_{\{|\xi| >  1  \}} m_\omega(|\xi| ) \cdot \hat g(\ell \xi) \; \overline{\hat h(\ell \xi)} \ \mathrm{d} \xi  \\
		&= \frac1\ell  q_0[g,h] -  \ell^2 \cdot \frac{\omega}{2}  \int_{\R^2}  \frac{ \hat g(\ell \xi) \; \overline{\hat h(\ell \xi)}}{|\xi|}   \ \mathrm{d} \xi + \ell^2 \int_{\R^2} m_{\omega,\mathrm{res}} (|\xi|) \cdot \hat g(\ell \xi) \; \overline{\hat h(\ell \xi)} \ \mathrm{d} \xi ,
	\end{align*}
	with 
	$$ m_{\omega,\mathrm{res}} (|\xi|) = -  \mathds 1_{\{ |\xi| \le   1\}} (\xi) \cdot |\xi|   + 
		\mathds 1_{\{|\xi| >  1  \}} (\xi) \cdot  \Bigr( m_\omega(|\xi|) -  |\xi|  \Bigr) +  
		\frac{\omega}{2|\xi|}   = \mathcal O(  |\xi|^{-3}) .  $$
	We choose the functions  $X_\omega, X_{\omega,\mathrm{res}} \in C^\infty(\R^2 \backslash \{0\})$ such that 
	\begin{align*}
		\hat X_\omega(\xi) = \frac{\omega}{4 \pi  |\xi|} \qquad \text{and} \qquad 
		 \hat X_{\omega,\mathrm{res}} (\xi) = \frac{1}{2\pi} \; m_{\omega,\mathrm{res}} (\xi ) . 
	\end{align*} 
	Calculating $X_\omega$ for   $(s,\varphi) \in \R_+ \times (0,2 \pi)$, $x = (s \cos \varphi, s \sin \varphi)$ we have 
	\begin{align*}
		X_\omega(x) &= \frac{\omega}{8\pi^2 } \int_{\R^2} \frac{\mathrm{e}^{\mathrm{i} x \xi}}{|\xi|}  \ \mathrm{d} \xi 
			= \frac{\omega}{8\pi^2} \int_0^\infty \int_{-\pi}^{\pi} \mathrm{e}^{\mathrm i s t \left( \cos \varphi \; \sin\varphi \right) \cdot 
			\left(\cos u \;  \sin u \right)^T} \ \mathrm{d} u \ \mathrm{d} t \\
			&= \frac{\omega}{4\pi} \int_0^\infty J_0(t s) \ \mathrm{d} t . 
	\end{align*}
	Here $J_0$ is the Bessel function of the first kind of order $0$. 
	Moreover, all integrals should be interpreted as oscillatory integrals or improper Riemann integrals. 
	Using \cite[Section 6.511]{GradRyz} we obtain 
	\begin{equation}
		X_\omega(x) = \frac{\omega}{4\pi |x|}  \int_0^\infty J_0(r) \ \mathrm{d} r = \frac\omega{4\pi  |x|} , 
	\end{equation}
	and thus, 
	\begin{align*}
		\ell^2 \cdot \frac{\omega}2 \int_{\R^2}  \frac{\hat g(\ell \xi) \; \overline{\hat h(\ell \xi)} }{|\xi|}   \ \mathrm{d} \xi
			&= \frac\omega{4\pi} \scal{K_{\frac1{|x|}}  T_\ell g  }{T_\ell h}_{\Sigma_\ell} 
		= \frac{\omega\cdot \ell^2} {4\pi} \int_{\Sigma \times \Sigma}
			\frac{g(x) \; \overline{h(z)}}{\ell |z-x|}  \ \mathrm{d} x \ \mathrm{d} z \\[5pt]
		&=\ell  \cdot \frac{\pi}{4 \alpha^2} \scal{K_{\frac1{|x|}}  g}{h}_{\Sigma} + 
			\mathcal O(\pi^2/ \alpha^2 - \omega) .
	\end{align*}	
	Note that  $m_{\omega,\mathrm{res}}(\xi)  = \mathcal O( |\xi|^{-3})$ as $|\xi| \to \infty$, uniformly in $\omega \in (0, \pi^2/\alpha^2)$, and thus, 
	$$ \sup_{\omega \in (0, \pi^2/\alpha^2)} \|X_{\omega, \mathrm{res}} \|_{L_\infty(\R^2)} < \infty , $$ 
	which implies 
	\begin{align*}
		\ell^2 \int_{\R^2} m_{\omega,\mathrm{res}} (|\xi|) \cdot \hat g(\ell \xi) \; \overline{\hat h(\ell \xi)} \ \mathrm{d} \xi  
		&= \ell^2 \int_{\Sigma \times \Sigma}  X_{\omega, \mathrm{res}} (\ell (z-x) ) \cdot  g(x) \; \overline{h(z)} \ \mathrm{d} x \ \mathrm{d} z = \mathcal O(\ell^2) .   
	\end{align*}
	This  concludes the proof of the theorem. 
\end{proof}
Let us now prove the asymptotics of the ground state eigenvalue of the operator $A_{\ell,b}$ as $\ell \to 0$. 
We shall omit the proof of the uniqueness or the existence of the eigenvalue for small $\ell >0$
as this  follows in much the same way as in Lemma \ref{lemma:uniqueness_2D}. 
We note that the operator $\mathcal Q_0$ is again invertible and  a Fredholm operator  since $\tilde H^{1/2}_0(\Sigma)$ is compactly embedded
into $L_2(\Sigma)$, cf.\ the arguments from the previous section.
Then for arbitrary $b \in L_\infty(\R)$ we have  
$$ \ell \mathcal Q_b(\ell, \omega)  = 
	\mathcal Q_0    + \ell^3 \cdot  \frac{|\Sigma|  }{4 \alpha^3} \ln (\pi^2/\alpha^2 - \omega ) \cdot P_\ct + \ell  R_b(\ell,\omega )  $$
with  
$$ \sup \{ \| R_b (\ell,\omega ) \| : \ell \in (0,\ell_0) \; \wedge \; | \omega - \pi^2/\alpha^2| < \varepsilon \} < \infty . $$
Applying the  Birman-Schwinger principle, we obtain the following  identity for the eigenvalue $\lambda(\ell)$
$$ - \frac{\ell^3 }{4 \alpha^3} \ln (\pi^2/\alpha^2 - \lambda(\ell) )  \cdot 
	\scal{( \mathcal Q_0 +\ell R_b(\ell, \lambda(\ell)))^{-1} \phi_0 }{\phi_0 }_\Sigma = 1   $$
or equivalently
$$  \ln (\pi^2/\alpha^2 - \lambda(\ell) )   
	= - \frac{4 \alpha^3}{\ell^3 \cdot \scal{(\mathcal Q_0 +\ell R_b(\ell, \lambda(\ell)))^{-1} \phi_0 }{\phi_0 }_\Sigma }
	  .  $$
Here $\phi_0 (x) =  1$ is again the non-normalised constant function in $L_2(\Sigma)$. 
As before we obtain 
\begin{align*}
 	- \ln (\pi^2/\alpha^2 - \lambda(\ell) ) &= \frac{4 \alpha^3}{\ell^3 \cdot \tau_1(\Sigma) + \mathcal O(\ell^4) }		
\end{align*}
as $\ell \to 0$. Here 
\begin{align}\label{def:tau_0_3D}
	\tau_1(\Sigma) := \scal{\mathcal Q_0^{-1} \phi_0}{\phi_0}_{\Sigma} = \scal{\mathcal Q_0^{-1/2} \phi_0 }{\mathcal Q_0^{-1/2} \phi_0}_{\Sigma} > 0  .
\end{align}
This proves the first term of the asymptotic formula. Higher terms of the expansion may be calculated as above; assuming smoothness of $b$ we obtain 
\begin{align*}
	\ln (\pi^2/\alpha^2 - \lambda(\ell) ) 
	&=  - \ell^{-3} \frac{4\alpha^3}{\tau_0(\Sigma)  - \ell \cdot  \tau_1(\Sigma) \cdot  b(0)  +  \mathcal O(\ell^3)}  \qquad \text{as} \quad \ell \to 0 ,
\end{align*}
 where 
\begin{align}\label{def:tau_1_3D} 
	\tau_1(\Sigma) = \scal{\mathcal Q_0^{-1} \phi_0}{\mathcal Q_0^{-1} \phi_0}_{\Sigma} >0  . 
\end{align}
This concludes the proof of Theorem \ref{th:main_3D}. 

\section*{\bf Acknowledgements} 
The research and in particular A.H. were supported by DFG grant WE-1964/4-1.

\bibliographystyle{plain}

\end{document}